\documentclass[a4paper, 11pt, reqno]{amsart}

\usepackage{amsmath, amssymb, amsthm, mathrsfs, hyperref, color, enumerate, tikz, tikz-cd, mathtools, mathdots,graphicx}
\usepackage[capitalize]{cleveref}
\usetikzlibrary{calc, matrix, shapes, decorations.pathreplacing, intersections}

\usepackage{verbatim}
\usepackage{caption}
\newif\ifshowtikz

\let\oldtikzcd\tikzcd
\let\oldendtikzcd\endtikzcd

\renewenvironment{tikzcd}{%
    \ifshowtikz\expandafter\oldtikzcd%
    \else\comment%
    \fi
}{%
    \ifshowtikz\oldendtikzcd%
    \else\endcomment%
    \fi
}

\let\oldtikzpicture\tikzpicture
\let\oldendtikzpicture\endtikzpicture

\renewenvironment{tikzpicture}{%
    \ifshowtikz\expandafter\oldtikzpicture%
    \else\comment%
    \fi
}{%
    \ifshowtikz\oldendtikzpicture%
    \else\endcomment%
    \fi
}

\showtikztrue

\usepackage{booktabs, array}
\newcolumntype{C}{>{$}c<{$}}
\renewcommand{\P}{\mathbb{P}}
\newcommand{\R}{\mathbb{R}}
\newcommand{\C}{\mathbb{C}}
\newcommand{\A}{\mathbb{A}}
\newcommand{\Z}{\mathbb{Z}}

\newcommand{\Hom}{\operatorname{Hom}}
\newcommand{\Ext}{\operatorname{Ext}}
\newcommand{\End}{\operatorname{End}}
\newcommand{\perf}{\operatorname{perf}}
\renewcommand{\S}{\mathcal{S}}
\renewcommand{\hom}{\operatorname{hom}}
\newcommand{\Spec}{\operatorname{Spec}}

\newcommand{\vc}[2]{{{}^{#2}V_{#1}}}
\newcommand{\vcc}[3]{{{}^{#2}V_{#1}^{#3}}}

\newcommand{\w}{\mathbf{w}}
\newcommand{\wt}{\check{\mathbf{w}}}
\newcommand{\W}{\mathbf{W}}

\newcommand{\Jac}{\text{Jac}}
\newcommand{\HH}{\text{HH}}
\newcommand{\SH}{\text{SH}}

\newcommand{\mf}{\text{mf}}
\newcommand{\coh}{\operatorname{coh}}
\newcommand{\Qcoh}{\operatorname{Qcoh}}
\newcommand{\mcg}{\mathcal{M}(\Sigma;\partial\Sigma)}

\theoremstyle{plain}

\newtheorem{mconj}{Conjecture}
\newtheorem{thm}{Theorem}[section]
\newtheorem{lem}[thm]{Lemma}

\newtheorem{cor}[thm]{Corollary}
\newtheorem{definition}{Definition}[section]
\theoremstyle{remark}

\newtheorem{rmk}[thm]{Remark}

\title{Homological mirror symmetry for invertible polynomials in two variables}
\address{Universit\"at Hamburg\\ Fachbereich Mathematik\\ Bundesstra{\ss}e 55\\ 20146 Hamburg\\ Germany}
\author{Matthew Habermann}
\email{matthew.habermann@uni-hamburg.de}

\begin{document}
\begin{abstract}
In this paper, we give a proof of homological mirror symmetry for two variable invertible polynomials, where the symmetry group on the B--side is taken to be maximal. The proof involves an explicit gluing construction of the Milnor fibres, and, as an application, we prove derived equivalences between certain nodal stacky curves, some of whose irreducible components have non-trivial generic stabiliser. 
\end{abstract}
\maketitle
\section{Introduction}
Consider an $n\times n$ matrix $A$ with non-negative integer entries $a_{ij}$. From this, we can define a polynomial $\w\in\C[x_1,\dots,x_n]$ given by
\begin{align*}
\w(x_1,\dots,x_n)=\sum_{i=1}^n\prod_{j=1}^nx_j^{a_{ij}}. 
\end{align*} 
In what follows $\w$ will always be quasi-homogeneous, and so we can associate to it a \emph{weight system} $(d_0,d_1,\dots, d_n;h)$, where 
\begin{align*}
\w(t^{d_1}x_1,\dots, t^{d_n}x_n)=t^h\w(x_1,\dots,x_n),
\end{align*}
and $d_{0}:=h-d_1-\dots-d_n$. In \cite{BerglundHubsch}, the authors define the transpose of $\w$, denoted by $\check{\w}$, to be the polynomial associated to $A^T$, 
\begin{align*}
\check{\w}(\check{x}_1,\dots,\check{x}_n)=\sum_{i=1}^n\prod_{j=1}^n\check{x}_j^{a_{ji}},
\end{align*}
and we call this the \emph{Berglund--H\"ubsch transpose}. One can associate a weight system for $\check{\w}$, denoted by $(\check{d}_{0},\check{d}_1,\dots,\check{d}_n;\check{h})$, in the same way. We call a polynomial $\w$ \emph{invertible} if the matrix $A$ is invertible over $\mathbb{Q}$, and if both $\w$ and $\check{\w}$ define isolated singularities at the origin (cf. Definition \ref{invertible poly}).\\

Recall that for $f\in\C[x_1,\dots, x_n]$ and $g\in \C[y_1,\dots, y_m]$, their Thom--Sebastiani sum is defined as 
\begin{align}\label{TSSum}
f\boxplus g=f\otimes 1+1\otimes g\in \C[x_1,\dots, x_n,y_1,\dots, y_m].
\end{align} 
A corollary of Kreuzer--Skarke's classification of quasi-homogeneous polynomials, \cite{kreuzer1992}, is that any invertible polynomial can be decoupled into the Thom--Sebastiani sum of \textit{atomic} polynomials of the following three types:
\begin{itemize}
	\item Fermat: $\w=x_1^{p_1}$,
	\item Loop: $\w=x_1^{p_1}x_2+x_2^{p_2}x_3+\dots+x_n^{p_n}x_1$,
	\item Chain: $\w=x_1^{p_1}x_2+x_2^{p_2}x_3+\dots+x_n^{p_n}$.
\end{itemize}
The Thom--Sebastiani sums of polynomials of Fermat type are also called Brieskorn--Pham.\\

To any invertible polynomial, one can associate its \textit{maximal symmetry group}
\begin{align}\label{MaxSymGp}
\Gamma_{\w}:=\{(t_1,\dots,t_{n+1})\in(\C^*)^{n+1}|\ \w(t_1x_1,\dots, t_nx_n)=t_{n+1}\w(x_1,\dots, x_n)\}.
\end{align}
Since the $t_{n+1}$ variable is uniquely determined by the other $t_i$, we will think of $\Gamma_{\w}$ as a subgroup of $(\C^*)^n$. It is a finite extension of $\C^*$, and is the group of diagonal transformations of $\A^n$ which keep $\w$ semi-invariant with respect to the character $(t_1\,\dots,t_{n+1})\mapsto t_{n+1}$. Homological Berglund--H\"ubsch mirror symmetry predicts:
\begin{mconj}\label{BH conj} For any invertible polynomial, $\w$, there is a quasi-equivalence 
	\begin{align*}
	\mathrm{mf}(\A^n,\Gamma_\w,\w)\simeq \mathcal{F}(\check{\w})
	\end{align*}
of pre-triangulated $A_\infty$-categories over $\C$.
\end{mconj}
In the above, $\mathrm{mf}(\A^n,\Gamma_\w,\w)$ is the category of $\Gamma_{\w}$--equivariant matrix factorisations of $\w$, and $\mathcal{F}(\check{\w})$ is the Fukaya--Seidel category associated to a Morsification of $\check{\w}$, as defined in \cite{SeidelBook}. Conjecture \ref{BH conj} goes back to \cite{Weightedprojectivelines} and \cite{2006math......4361U}, and there have recently been many results in the direction of establishing it. It has been proven in several cases -- in particular, for Brieskorn--Pham polynomials in any number of variables in \cite{Futaki2011}, and for Thom--Sebastiani sums of polynomials of type $A$ and $D$ in \cite{FU2}. Conjecture \ref{BH conj} is also established for all invertible polynomials in two variables in \cite{HabermannSmith}. For each class of invertible polynomial, recent work of Kravets (\cite{2019arXiv191109859K}) establishes a full, strong, exceptional collection for $\mathrm{mf}(\A^n,\Gamma_\w,\w)$ with $n\leq 3$. In the case of chain polynomials in any number of variables, Hirano and Ouchi (\cite{HiranoOuchi}) show that the category $\mathrm{mf}(\A^n,\Gamma_\w,\w)$ has a tilting object, and a full, strong, exceptional collection whose length is the Milnor number of $\check{\w}$. For further discussion and background on Conjecture \ref{BH conj}, see \cite{2016arXiv160106027E}, and references therein.\\

There is also an extension of Conjecture \ref{BH conj} (\cite{10.1093/imrn/rns115}, \cite{Weightedprojectivelines}, \cite{Krawitz}), where rather than considering the maximal symmetry group, one considers certain subgroups of finite index. Correspondingly, one must then consider an `orbifold Fukaya--Seidel' category, which incorporates a dual group in its data. Recently, this generalised conjecture which takes into account a symmetry group on the A--side was established in the $\Z/2$-graded case for two variable invertible polynomials in \cite{ChoChoaJeong}.\\

The main focus of this paper is homological mirror symmetry where the (completion of the) Milnor fibre of $\wt$, 
\begin{align}\label{MilnorFibre}
\check{V}_{\wt}:=\wt^{-1}(1)
\end{align}
is taken as the A--model. On the B--side, one extends the action of $\Gamma_{\w}$ to $\A^{n+1}$ in a natural way, as described in Section \ref{background section} for the case of $n=2$. The Lekili--Ueda conjecture predicts:
\begin{mconj}[{\cite[Conjecture 1.4]{2018arXiv180604345L}}]\label{LUConj}
For any pair of invertible polynomials $\w$, $\wt$, there is a quasi-equivalence
\begin{align*}
\mathcal{W}(\check{V}_{\wt})\simeq\mathrm{mf}(\A^{n+1},\Gamma_{\w},\w+x_0x_1\dots x_n)
\end{align*}
of pre-triangulated $A_\infty$-categories over $\C$. 
\end{mconj}
In the above, $\mathcal{W}(\check{V}_{\wt})$ is the wrapped Fukaya category of the Milnor fibre of $\wt$. This category is completed with respect to cones and direct summands, as stated in Section \ref{convenctions}.  Recently, Conjecture \ref{LUConj}  was established for simple singularities in any dimension in \cite{2020arXiv200407374L}, and has subsequently been established in the $\Z/2$-graded case in \cite{Gammage}. \\

There is a trichotomy of cases depending on whether the weight $d_0$ is positive, negative, or zero. In the log general type case of $d_0>0$, there is a quasi-equivalence
\begin{align}
\mathrm{mf}(\A^{n+1},\Gamma_{\w},\w+x_0x_1\dots x_n)\simeq \coh Z_{\w},
\end{align}
where 
\begin{align}
Z_\w:=\big[\big(\Spec\C[x_0,x_1,\dots,x_n]/(\w+x_0x_1\dots x_{n})\setminus(\boldsymbol{0})\big)/\Gamma_\w\big].
\end{align}
This equivalence is a generalisation of (\cite[Theorem 3.11]{Orlov2009}), where it was proven in the context of triangulated categories, and where $\C^*\simeq \Gamma_{\w}$. The generalisation to the case where $\Gamma_\w$ is a finite extension of $\C^*$ is straightforward, and the extension to the setting of dg-categories was studied in \cite{Shipmandg}, \cite{Isikorlov}, \cite{CaldararuTu}. The main focus of this paper is the case of curves, for which the only invertible polynomial which is not of log general type is $x^2+y^2$. This, however, corresponds to the well-understood HMS statement for $\C^*$. We will therefore restrict ourselves to the log general type case for the remainder of the paper.\\

Recall that, since $Z_\w$ is a proper stack, the subcategory $\perf Z_\w\subseteq \coh Z_\w$ consists precisely of Ext-finite objects, meaning that $X\in \perf Z_\w$ if and only if $\bigoplus_{i\in\Z}\Ext^i(X,Y)$ is finite dimensional for every object $Y\in\coh Z_\w$. On the symplectic side of the correspondence, it is clear that compact Lagrangians can have morphisms in only finitely many degrees with any other Lagrangian, but it is not known that this is necessarily not true for non-compact Lagrangians. This is reasonable to expect, however, and is certainly true in every known case. Therefore, in the log general type case, one expects that Conjecture \ref{LUConj} implies
\begin{align}\label{PerfFukayaEquiv}
\mathcal{F}(\check{V}_{\wt})\simeq\perf Z_{\w}.
\end{align}
Establishing this quasi-equivalence in the case of curves is the main result of this paper. 
\begin{thm}\label{main theorem}
	Let $\w$ be an invertible polynomial in two variables, and $\check{\w}$ its transpose. Then there is a quasi-equivalence 
	\begin{align*}
	\mathcal{F}(\check{V}_{\check{\w}})\simeq\perf Z_\w
	\end{align*}
	of $\Z$-graded pre-triangulated $A_\infty$-categories over $\C$, where $\mathcal{F}(\check{V}_{\check{\w}})$ and $\perf Z_\w$ are as in Section \ref{convenctions}.
\end{thm}
\begin{rmk}
It should be reiterated that, although there is a trichotomy of cases depending on the weight $d_0$, all but one invertible polynomials in two variables are of log general type, and this exception is well-understood. We are therefore free to state Theorem \ref{main theorem} in the context of invertible polynomials of log general type without making any assumptions on $d_0$. 
\end{rmk}
The first instances of the quasi-equivalence in \eqref{PerfFukayaEquiv} were established for the case of $\w=x^2+y^3$ in \cite{2012arXiv1211.4632L}, and the cases of $\w=x^3y+y^2$, $\w=x^3+y^3$, and $\w=x^4+y^2$ in \cite{Auslanderorders}. It was also established in \cite{2018arXiv180604345L} for $\w=\sum_{i=1}^n x_i^{n+1}$, and $\w=x_1^2+\sum_{i=2}^n x_i^{2n}$, both for $n>1$. \\

In \cite{Lekili2019}, the authors use mirror symmetry arguments to deduce derived equivalences between rings of certain nodal stacky curves. We elaborate on these arguments in order to identify which Milnor fibres are graded symplectomorphic, and this enables us to deduce derived equivalences between nodal stacky curves with different numbers of irreducible components, some of which have non-trivial generic stabiliser. 
\begin{cor}\label{HMS corollary}
	For each $n\geq 1$, $q \geq 2$ let $\w_{\mathrm{loop}}=x^{n(q-1)+1}y+y^qx$, $\w_{\mathrm{chain}}=x^{nq+1}y+y^q$, each with maximal symmetry group. We then have quasi-equivalences
	\begin{align*}
	\perf Z_{\w_{\mathrm{loop}}}&\simeq   \perf Z_{\w_{\mathrm{chain}}}
	\end{align*}
	of pre-triangulated $A_\infty$-categories over $\C$. Similar, for $n\geq 1$ and $p\geq 2$ or $n\geq p=2$, let $\w_{\mathrm{chain}}'=x^{p}y+y^{n(p-1)}$, $\w_{\mathrm{BP}}=x^{p}+y^{np}$, each with maximal symmetry group. Then, we have quasi equivalences 
	\begin{align*}
	\perf Z_{\w_{\mathrm{chain}}'}&\simeq   \perf Z_{\w_{\mathrm{BP}}}
	\end{align*}
	of pre-triangulated $A_\infty$-categories over $\C$.
\end{cor}
This is obtained by first proving that the Milnor fibres corresponding to the relevant Berglund--H\"ubsch transposes are graded symplectomorphic. This implies that their Fukaya categories are quasi-equivalent, which by Theorem \ref{main theorem} proves that the derived categories of perfect complexes of their mirrors are too. This corollary also appears as a special case of \cite[Corollary 5.15]{FAVERO2019943}, although was obtained there by a variation of GIT argument (\cite{BFKVGIT}, \cite{HLVGIT}).
\subsection{Strategy of proof}\label{IntroStrategy}
Our strategy follows that of \cite{2018arXiv180604345L}, where one reduces the proof of Theorem \ref{main theorem} to a deformation theory argument. For the case at hand, this approach is predicated on the proof of Conjecture \ref{BH conj} for curves given in \cite{HabermannSmith}.\\

On the A--side of the correspondence, we have that there is a restriction functor
\begin{align}\label{FSRestriction}
\begin{split}
\mathcal{F}(\wt)&\rightarrow \mathcal{F}(\check{V}_{\wt})\\
S^{\rightarrow}&\mapsto \partial S^{\rightarrow}=:S,
\end{split}
\end{align} 
where we equip the vanishing cycle $\partial S^{\rightarrow}$ with the induced (non-trivial) spin structure. Suppose that $(S_i^\rightarrow)_{i=1}^{\check{\mu}}$ is a collection of thimbles which generates $\mathcal{F}(\check{\w})$, where $\check{\mu}$ is the Milnor number of $\wt$, and that $\mathcal{S}^\rightarrow$ is the full subcategory of $\mathcal{F}(\check{\w})$ whose objects are $(S_i^\rightarrow)_{i=1}^{\check{\mu}}$. Denote its $A_\infty$-endomorphism algebra by
\begin{align}\label{DirectedAlg}
\mathcal{A}^\rightarrow:=\bigoplus _{i,j}^{\check{\mu}}\hom_{\mathcal{F}(\check{\w})}(S_i^\rightarrow,S_j^\rightarrow),
\end{align}
and its cohomology algebra $A^{\rightarrow}:=H^*(\mathcal{A}^{\rightarrow})$. Correspondingly, let $\mathcal{S}$ be the collection $(S_i)_{i=1}^{\check{\mu}}$ of vanishing cycles equipped with the non-trivial spin structure, considered as a full subcategory of the compact Fukaya category of the Milnor fibre, and $\mathcal{A}$ its $A_\infty$-endomorphism algebra. Poincar\'e duality tells us that we can identify $H^*(\mathcal{A})$  with
\begin{align}\label{TEA VS}
A:=A^\rightarrow\oplus(A^\rightarrow)^\vee[1-n]
\end{align}
 as a vector space. In our case, we will deduce in Section \ref{generators and formality} that the algebra structure on $A$ is induced purely from the $A^\rightarrow$--bimodule structure of $(A^\rightarrow)^\vee[1-n]$. Namely, we have
\begin{align}\label{TEA multiplication}
(a,f)\cdot (b,g)=(ab,ag+fb).
\end{align}
This is known as a trivial extension algebra of degree $n-1$. By the argument of \cite[Lemma 5.4]{HMSQuartic}, when the weight $\check{d}_{0}\neq0$, $\mathcal{S}$ split generates the compact Fukaya category of the Milnor fibre. Therefore, in order to characterise this category, it is sufficient to identify the $A_\infty$-structure on $A$ which is given by $\mathcal{A}$, up to gauge transformation (a.k.a. formal diffeomorphism).\\
 
On the algebro-geometric side of the correspondence, one can consider the Jacobi algebra, 
\begin{align}\label{Jac alg}
\text{Jac}_{\w}=\C[x_1,\dots,x_n]/(\partial_1\w,\dots,\partial_n\w).
\end{align}
Since the singularity is isolated, this algebra has dimension $\mu<\infty$, the Milnor number of $\w$. Let $J_\w$ be the set of exponents for a basis of this algebra, and consider the semi-universal unfoldings of $\w$,
\begin{align}\label{semiuniv unfolding}
\widetilde{\w}:=\w+\sum_{\textbf{j}\in J_\w}u_{\textbf{j}}x_1^{j_1}\dots x_n^{j_n}.
\end{align}
Such unfoldings are universal in the sense that every other unfolding of $\w$ is induced from $\widetilde{\w}$ by a change of coordinates; however, this change of coordinates is not unique. These semi-universal unfoldings are parametrised by $\mu$ complex parameters, and we set 
\begin{align}\label{UDef}
U:=\Spec\C[u_1,\dots, u_\mu].
\end{align}
 We can therefore consider $\widetilde{\w}$ as a map
\begin{align}\label{seminuniv unfolding II}
\widetilde{\w}:\A^n\times U\rightarrow\A^1,
\end{align}
and define 
\begin{align}\label{wuDef}
\w_u:=\widetilde{\w}|_{\A^n\times\{u\}}.
\end{align}
 To such a polynomial $\w_u$, we associate a stack $V_u$, defined in the case of two variables in \eqref{OpenSubstackdefn}. In the case where the weight $d_{0}>0$, we want to compactify $V_u$ to a Calabi--Yau hypersurface in a quotient of weighted projective space by a finite group, although this is not possible for every $u\in U$. As previously mentioned, we extend the action of $\Gamma_{\w}$ to $\A^{n+1}$ in a prescribed way, and define $U_+\subseteq U$ to be the subspace such that $\w_u$ can be quasi-homogenised to $\W_u\in  \C[x_0,x_1,\dots,x_n]$ with respect to this action. Following \cite{2018arXiv180604345L}, one then defines
\begin{align}\label{Yu defn}
Y_u:=\big[(\W_u^{-1}(0)\setminus(\textbf{0}))/\Gamma_{\w}\big]
\end{align} 
for each $u\in U_+$. It goes back to the work of Pinkham (\cite{pinkham1974deformations}), that the fact that $\w$ is quasi-homogeneous forces there to be a $\C^*$-action on $U_+$. We therefore have that $Y_u\simeq Y_v$ if and only if $v=t\cdot u$ for some $t\in \C^*$. By construction, the dualising sheaf of this stack is trivial and it is a compactification of $V_u$.\\

For each $u\in U_+$, there is a functor
\begin{align}\label{intro pushforward functor}
\mf(\A^n,\Gamma_\w,\w)\rightarrow \coh Y_u
\end{align}
which is to be expounded upon in Section \ref{background section} for the case of curves. In any case where $\mf(\A^n,\Gamma_\w,\w)$ has a tilting object, $\mathcal{E}$, denote by $\mathcal{S}_u$ the image of $\mathcal{E}$ by \eqref{intro pushforward functor}. It is then a theorem of Lekili and Ueda (\cite[Theorem 4.1]{2018arXiv180604345L}) that $\mathcal{S}_u$ split-generates $\perf Y_u$. Let $\mathcal{A}_u$ be the minimal $A_\infty$-endomorphism algebra of $\mathcal{S}_u$. Then, by the work of Ueda in \cite{Hyperplanesections}, we have that $A_u:=H^*(\mathcal{A}_u)$ is also given by the degree $n-1$ trivial extension algebra of the endomorphism algebra of $\mathcal{E}$, and is, in particular, independent of $u$. In the case where Conjecture \ref{BH conj} is solved by exactly matching generators, as in \cite{HabermannSmith}, we have that, at the level of cohomology, the endomorphism algebra of the generators on both the A--, and B--sides are given by the same algebra, which we denote by $A$.  In light of this, establishing the equivalence \eqref{PerfFukayaEquiv} boils down to identifying the $A_\infty$-structure given by the chain level endomorphism algebra on the B--side which matches with that of the A--side. With this perspective, homological mirror symmetry for invertible polynomials turns into a deformation theory problem.\\

Recall that for a graded algebra, $A$, the Hochschild cochain complex has a bigrading. Namely, we consider $\text{CC}^{r+s}(A,A)_s$ to be the space of maps $A^{\otimes r}\rightarrow A[s]$. In general, if $\mu^\bullet$ is a minimal $A_\infty$-structure on $A$, then deformations which keep $\mu^k$ for $1\leq k\leq m$ fixed are controlled by $\bigoplus_{i>m-2} \HH^2(A)_{-i}$ (see, for example, \cite[Section 3a]{HMSQuartic}). In particular, the deformations of $A$ to a minimal $A_\infty$-model with prescribed $\mu^2$ are controlled by $\HH^2(A)_{<0}=\bigoplus_{i>1}\HH^2(A)_{-i}$. Furthermore, note that $\HH^2(A)_0$ is the first order deformations of the algebra structure on $A$. It is natural to consider the functor which takes an algebra to the set of gauge equivalence classes of $A_\infty$-structures on that algebra. It is a theorem of Polishchuk (\cite[Corollary 3.2.5]{polishchuk2017}) that if $\HH^1(A)_{<0}=0$, then this functor is represented by an affine scheme, $\mathcal{U}_\infty(A)$. Moreover, if $\dim \HH^2(A)_{<0}<\infty$, then \cite[Corollary 3.2.6]{polishchuk2017} shows that this scheme is of finite type. This functor was first studied in the context of homological mirror symmetry in \cite{LekiliPolishchuk}. There is a natural $\C^*$-action on $\mathcal{U}_\infty(A)$ given by sending $\{\mu^k\}_{k=1}^\infty$ to $\{t^{k-2}\mu^k\}_{k=1}^\infty$, and this is denoted by $\mathcal{A}\mapsto t^*\mathcal{A}$. Note that the formal $A_\infty$-structure is the fixed point of this action. For each $t\neq 0$, we have that $\mathcal{A}$ and $t^*\mathcal{A}$ are quasi-isomorphic, although not through a gauge transformation (\cite[Section 3]{HMSQuartic}). \\

Now, for each $u\in U_+$, we have that $\mathcal{A}_u$ defines an $A_\infty$-structure on $A$ with $\mu^2$ given as in \eqref{TEA multiplication}. Therefore, it defines a point in $\mathcal{U}_\infty(A)$, and so we get a map
\begin{align}\label{U+ModuliMap}
U_+\rightarrow \mathcal{U}_\infty(A).
\end{align}
If we can show that \eqref{U+ModuliMap} is an isomorphism, then we know that every $A_\infty$-structure on $A$ is realised as the $A_\infty$-endomorphism algebra of $\mathcal{S}_u$ for some $u\in U_+$.  In the case that the pair $(\w,\Gamma)$ is untwisted (see Definition \ref{twisting}), we have by a theorem of Lekili and Ueda (\cite[Theorem 1.6]{2018arXiv180604345L}) that there is a $\C^*$-equivariant isomorphism of affine varieties $U_+\xrightarrow{\sim}\mathcal{U}_\infty(A)$ which sends the origin to the formal $A_\infty$-structure. By removing the fixed point of the action on both sides, we have that this isomorphism descends to an isomorphism 
\begin{align}
\big(U_+\setminus(\boldsymbol{0})\big)/\C^*\xrightarrow{\sim} \big(\mathcal{U}_\infty(A)\setminus(\boldsymbol{0})\big)/\C^*=:\mathcal{M}_{\infty}(A).
\end{align}
Therefore, in the case where $\w$ is untwisted, we have that, up to scaling, there is some $u\in U_+$ for which \eqref{PerfFukayaEquiv} holds.\\

We end this section by briefly remarking that the moduli of $A_\infty$-structures argument employed in this paper fits into a broader framework which has proven to be a fruitful approach to HMS, and whose scope is more wide-reaching than that of invertible polynomials. In \cite{DehnSurgery} and\cite{2012arXiv1211.4632L}, the authors establish HMS for the once punctured torus by studying the moduli space of $A_\infty$-structures on the degree one trivial extension algebra of the $A_2$ quiver. Interestingly, it was proven that $\mathcal{M}_{\infty}(A)\simeq \overline{\mathcal{M}}_{1,1}$, the moduli space of elliptic curves. Further connection was made to the moduli theory of curves in \cite{Lekili2017AMC}, where the authors show the moduli space of $A_\infty$-structures on a particular algebra coincides with the modular compactification of genus $1$ curves with $n$ marked points, as constructed in \cite{Smyth}. This then leads them to prove homological mirror symmetry for the $n$-punctured torus in \cite{LekiliPolishchuk}.
\subsection{Structure of paper} In Section 2, we recall some basic facts about invertible polynomials in two variables, as well as compute $U_+$ in the relevant cases. In Section 3, we study the symplectic topology of the Milnor fibre. In Section 4, we compute the relevant Hochschild cohomology for invertible polynomials in two variables. In Section 5, we recall some facts about generators and formality for Fukaya categories and the proper algebraic stacks under consideration. Section 6 is then a proof of Theorem \ref{main theorem} and Corollary \ref{HMS corollary}.

\subsection{Conventions}\label{convenctions} Throughout this paper all Fukaya categories will be completed with respect to cones and direct summands. We will also denote the bounded derived category of coherent sheaves, its full subcategory consisting of perfect complexes, and the unbounded derived category of quasi-coherent sheaves on an algebraic stack $X$ by $\coh X$, $\perf X$, and $\Qcoh X$, respectively. For a dg-category $\mathcal{A}$, we will also denote the unbounded derived category of right dg-modules as $\text{Mod}\, \mathcal{A}$. All coefficient groups will be taken to be $\Z$ unless stated otherwise. By $\Z_n$ we mean $\Z/n\Z$, and by $\Z_{(2)}$ we mean the local ring of rational numbers with odd denominator. 
\subsection{Acknowledgements} The author would like to thank his Ph.D. supervisor Yank\i\ 
Lekili for suggesting the project, his guidance throughout, and for careful reading of
previous versions of this paper. He would also like to thank Jack Smith for his interest
in the project, and his valuable feedback. The author is grateful to the anonymous
referee for their helpful feedback and suggestions. This work was supported by the
Engineering and Physical Sciences Research Council [EP/L015234/1], The EPSRC
Centre for Doctoral Training in Geometry and Number Theory (The London School
of Geometry and Number Theory), University College London.

 \section{Invertible polynomials in two variables}\label{background section}

In this paper, we will focus on invertible polynomials in two variables, as well as their unfoldings and quasi-homogenisations. As such, we will restrict ourselves to this case in the rest of the paper, and consider the variables $x, y, z$. The purpose of this section is to give a self-contained overview of the required background on invertible polynomials in two variables, and then to calculate the relevant spaces of semi-universal unfoldings.  \\

As in the introduction, let $A=\begin{pmatrix}
i_1 & j_1\\
i_2 & j_2
\end{pmatrix}$ be a matrix with non-negative integer coefficients such that $\det A\neq 0$, and 
\begin{align*}
\w(x,y)=x^{i_1}y^{j_1}+x^{i_2}y^{j_2}
\end{align*}
the corresponding polynomial, with weight system $(d_0,d_1,d_2; h)$. Denote its Berglund--H\"ubsch transpose by $\check{\w}$, with corresponding weight system $(\check{d}_0,\check{d}_1,\check{d}_2;\check{h})$.  We will always assume that $\gcd(d_1,d_2,h)=\gcd(\check{d}_1,\check{d}_2,\check{h})=1$. Note that $d_0>0$ if and only if $\check{d}_0>0$.
\begin{definition}\label{invertible poly}
Let $A$ be a $2\times 2$ matrix with non-negative integer coefficients. Let $\w$  and $\check{\w}$ be as above. We call $\w$ an \emph{invertible polynomial} if $A$ is invertible over $\mathbb{Q}$, and $\w$ and $\check{\w}$ both have isolated singularities at the origin. 
\end{definition}
In what follows, we will always assume that $p$ and $q$ are always at least 2. For Brieskorn--Pham and chain polynomials, this is necessary for the origin to be a critical point of both $\w$ and $\check{\w}$. In the loop case, if one of $p$ or $q$ is 1, then one can see that $\w$ and $\check{\w}$ are equivalent to $x^2+y^2$ and $\check{x}^2+\check{y}^2$ by a change of variables. \\

The maximal symmetry group is defined as in \eqref{MaxSymGp}, and to each $t_i$ we associate a character given by 
\begin{align}
(t_1,t_2,t_3)\mapsto t_i.
\end{align}
The group of characters for $\Gamma_\w$ is given by 
\begin{align}\label{MaxSymGpCharacters}
\hat{\Gamma}_\w:=\big(\Z\chi_1\oplus\Z\chi_2\oplus\Z\chi_{3})/(i_k\chi_1+j_k\chi_2-\chi_{3})_{k\in\{1,2\}}.
\end{align}
Let $\chi_{\w}:=\chi_{3}$, so that the elements of $\Gamma_{\w}$ are the diagonal transformations of $\A^2$ which keep $\w$ semi-invariant with respect to $\chi_{\w}$,
\begin{align*}
\w(t_1x,t_2y)=\chi_{\w}(t_1,t_2)\w(x,y).
\end{align*}
The subgroup $\ker\chi_{\w}$ of $\Gamma_{\w}$ are those elements which keep $\w$ invariant, and this is called the \emph{maximal diagonal symmetry group}. There is an injective map 
\begin{align}\begin{split}
\phi:\C^*&\rightarrow \Gamma_\w\\
t&\mapsto (t^{d_1},t^{d_2}),
\end{split}
\end{align}
and this fits into the short exact sequence
\begin{align}
1\rightarrow\C^*\xrightarrow{\phi}\Gamma_\w\rightarrow \text{ker}\ \chi_\w /\langle j_\w\rangle\rightarrow 1,
\end{align}
where $j_\w$ generates the cyclic group $\text{im}(\phi)\cap\ker\chi_{\w}$, and is called the grading element. Let $\Gamma\subseteq \Gamma_\w$ be a subgroup of finite index containing $\phi(\C^*)$, and for each $\Gamma$ we denote $\chi:=\chi_{\w}|_{\Gamma}$. The statements of Theorem \ref{main theorem} and Corollary \ref{HMS corollary} require $\Gamma=\Gamma_\w$, since this avoids the problem of needing `orbifold Fukaya(--Seidel) categories', as described in the introduction. Nevertheless, we will use $\Gamma$ when what we say is valid for any $\Gamma\subseteq \Gamma_{\w}$, and $\Gamma_{\w}$ when we specifically mean the maximal symmetry group. \\

The Jacobi algebra of $\w$ with Milnor number $\mu$ is given in \eqref{Jac alg}. Let $J_\w$ be as in the introduction, and semi-universal unfoldings of $\w$ be as in \eqref{semiuniv unfolding}.  Let $U$ and $\w_u$ be are as in \eqref{UDef} and \eqref{wuDef}, respectively. As already noted, Pinkham (\cite{pinkham1974deformations}) observed that $\w$ being quasi-homogeneous means that the space $U$ comes with a natural $\C^*$-action on it. Namely, the action on $u_{ij}$ is given by $t\cdot u_{ij}=t^{h-d_1i-d_2j}u_{ij}$. For a fixed $u\in U$, define $\overline{R}_u:=\C[x,y]/(\w_u)$, and observe that by scaling $x,y$, one can identify $\overline{R}_u\simeq \overline{R}_{t\cdot u}$ for $t\in \C^*$. The origin is the only fixed point of this action.\\

For a fixed $\Gamma\subseteq\Gamma_\w$, we would like to quasi-homogenise $\w_u$. In order to do this, however, we will need to extend the action of $\Gamma$ to $\A^3$. The action on the $z$ variable is chosen by setting
\begin{align}\label{t0 weight}
\chi_0(t_1,t_2)=\chi(t_1,t_2)t_1^{-1} t_2^{-1}.
\end{align}
This is done precisely so that $x^\vee\wedge y^\vee\wedge z^\vee$ is isomorphic to $\chi$ as a $\Gamma$-module. With this weight, we want to restrict ourselves to the subspace $U_+\subseteq U$ for which $\w_u$ is quasi-homogenisable, and has only positive powers of $z$. We define $U_+$ to be the subset of $u_{ij}$ in $U$ which can be non-zero only if there exists a positive integer $w_{ij}$ such that
\begin{align}\label{admissible criterion}
\chi^{w_{ij}-1}=t_1^{w_{ij}-i}t_2^{w_{ij}-j},
\end{align}
and consider $\W_u$ to be the quasi-homogenisation of $\w_u$ for each $u\in U_+$. Let $J_+\subseteq J_{\w}$ be the subset satisfying this condition.\\

For a fixed $u\in U_+$, we set $R_u:=\C[x,y,z]/(\W_u)$. By an abuse of notation, we will also denote the pullback of $\w$ to $\A^{3}$ by $\w$. We have that $Y_u$ is defined as in \eqref{Yu defn}, and each $Y_u$ is the compactification of
\begin{align}\label{OpenSubstackdefn}
V_u:=\big[\big(\Spec\overline{R}_u\setminus(\boldsymbol{0})\big)\big/\ker\chi_0\big],
\end{align}
and the divisor at infinity $X_u=Y_u\setminus V_u$ is isomorphic to $X=\big[\big(\Spec\overline{R}_0\setminus(\boldsymbol{0})\big)\big/\Gamma\big]$ for each $u\in U_+$. The condition $d_0>0$ ensures that each $Y_u$ is a proper stack. \\

These $\W_u$ fit together to form a family
\begin{align*}
\W_+:=\w(x,y)+\sum_{(i,j)\in J_+}u_{ij}x^{i}y^{j}z^{w_{ij}}:\A^3\times U_+\rightarrow \A^1
\end{align*}
such that $\W_u:=\W_+|_{\A^{3}\times \{u\}}$. Following \cite{2018arXiv180604345L}, we can then define 
\begin{align*}
\mathcal{Y}:=\big[\big(\W_+^{-1}(0)\setminus\big(\textbf{0}\times U_+\big)\big)\big/\Gamma\big],
\end{align*}
and this gives us a family
\begin{align*}
\pi_{\mathcal{Y}}:\mathcal{Y}\rightarrow U_+
\end{align*}
of stacks over $U_+$ such that $\pi_{\mathcal{Y}}^{-1}(u)=Y_u$ for each $u\in U_+$. Note that since each fibre is the compactification of $V_u$ by $X$, and $V_u\simeq V_{t\cdot u}$ for $t\in \C^*$, we have that the fibres above points in the same $\C^*$-orbit of $U_+$ are isomorphic. Furthermore, the relative dualising sheaf of this family is $\Gamma$-equivariantly trivial, by construction, and since $d_0>0$, this trivialisation is unique up to scaling.  \\

The map $R_u\rightarrow R_u/(z)\simeq\overline{R}_0$ induces a pushforward functor 
\begin{align}\label{PushforwardFunctor}
\mf(\A^2,\Gamma,\w)\rightarrow\mf(\A^3,\Gamma,\W_u)
\end{align}
obtained by considering the 2-periodic free resolution of an $\overline{R}_0$--module, and replacing each free $\overline{R}_0$ module with the $R_u$--free resolution 
\begin{align*}
0\rightarrow R_u(-\vec{z})\xrightarrow{z}R_u\rightarrow \overline{R}_0\rightarrow 0. 
\end{align*}
This is explained in detail, and in far greater generality, in \cite[Section 3]{Hyperplanesections}. \\

For the quotient stack $Y_u$, since the dualising sheaf of $Y_u$ is trivial for each $u\in U_+$, we have the Orlov equivalence
\begin{align}\label{coh mf quasi equiv}
\mf(\A^{3},\Gamma,\W_u)\simeq \coh Y_u.
\end{align}
The composition of \eqref{PushforwardFunctor} and Orlov equivalence gives the functor \eqref{intro pushforward functor}.

 \subsection{Unfoldings of loop polynomials}\label{UnfoldingsofLoopPolys}
 In the case of a two variable loop polynomial\\ $\w=x^py+y^qx$, we have $\mu=pq$, and 
 \begin{align}\label{LoopWeights}
 (d_1,d_2;h)=(\frac{q-1}{d},\frac{p-1}{d};\frac{pq-1}{d}),
 \end{align}
 where $d:=\gcd(p-1,q-1)$. Without loss of generality, we can assume that $p\geq q$. One has that 
 \begin{align}\label{Jac loop}
 \Jac_\w=\text{span}\{1,x,\dots,x^{p-1}\}\otimes\text{span}\{1,y,\dots,y^{q-1}\},
 \end{align}
 and
 \begin{align}
 \begin{split}\label{LoopMaxSymGp}
 \Gamma_{\w}=\big\{(t_1,t_2)\in(\C^*)^2|\  t_1^pt_2=t_2^qt_1\big\}&\xrightarrow{\sim}\C^*\times{\mu}_d\\
 (t_1,t_2)&\mapsto (t_1^nt_2^m,t_1^{\frac{p-1}{d}}t_2^{-\frac{q-1}{d}}),
  \end{split}
 \end{align}
 where $m,n$ is a fixed solution to
 \begin{align}\label{LoopBezout}
 m(p-1)+n(q-1)=d.
 \end{align}
 The image of the injective homomorphism 
 \begin{align*}
 \phi:\C^*&\rightarrow \Gamma_{\w}\\
 t&\mapsto (t^{\frac{q-1}{d}},t^{\frac{p-1}{d}})
 \end{align*}
 is an index $d$ subgroup of $\Gamma_\w$; however, we will only be interested in the maximal symmetry group, i.e. $\Gamma=\Gamma_\w$. A semi-universal unfolding is given by
 \begin{align}\label{LoopUnfolding}
 \widetilde{\w}(x,y)=x^py+y^qx+\sum_{\substack{0\leq i\leq p-1\\
 		0\leq j\leq q-1}}u_{ij}x^iy^j.
 \end{align}
By definition, $U_+$ is the subspace of $U$ containing elements such that there exists a positive integer $w_{ij}$ such that 
 \begin{align*}
 (t_1^pt_2)^{w_{ij}-1}=t_1^{w_{ij}-i}t_2^{w_{ij}-j}.
 \end{align*}
 There are three possibilities for $U_+$:\\
 
 {Case I:}
 For $q>2$ the only solution to this is $i=j=w_{ij}=1$, and so $U_+=\Spec\C[u_{11}]=\A^1$.\\
 
 {Case II:} $p>q=2$, we have $i=j=w_{ij}=1$, as well as $j=0$, $i=1$, and $w_{ij}=2$, and so $U_+=\Spec\C[u_{1,0},u_{1,1}]=\A^2$. \\ 
 
 {Case III:} When $p=q=2$, we have $i=j=w_{ij}=1$, $j=0$, $i=1$, $w_{ij}=2$, $j=1$, $i=0$, $w_{ij}=2$, as well as $i=j=0$, $w_{ij}=3$, and so $U_+=\Spec\C[u_{0,0},u_{1,0}, u_{0,1},u_{1,1}]=\A^4$.
 \subsection{Unfoldings of chain polynomials}\label{UnfoldingsofChainPolys}
 In the case of a two variable chain polynomial\\ $\w=x^py+y^q$, we have $\mu=pq-q+1$, and 
 \begin{align}\label{ChainWeights}
 (d_1,d_2;h)=(\frac{q-1}{d},\frac{p}{d};\frac{pq}{d}),
 \end{align}
 where $d:=\gcd(p,q-1)$. 
 \begin{rmk}
 	It should be stressed that this is the Milnor number on the \textit{B--side}. In the loop and Brieskorn--Pham cases the matrices defining the polynomials are symmetric, and the Milnor numbers of both sides will be the same, but this is not the case for chain polynomials. 
 \end{rmk}
 One has that 
 \begin{align}\label{Jac chain}
 \Jac_\w=\text{span}\{1,x,\dots,x^{p-2}\}\otimes\text{span}\{1,y,\dots,y^{q-1}\}\oplus\text{span}\{x^{p-1}\},
 \end{align}
 and
 \begin{align}
 \begin{split}\label{ChainMaxSymGp}
 \Gamma_{\w}=\big\{(t_1,t_2)\in(\C^*)^2|\  t_1^pt_2=t_2^q\big\}&\xrightarrow{\sim}\C^*\times{\mu}_d\\
 (t_1,t_2)&\mapsto (t_1^nt_2^m,t_1^{\frac{p}{d}}t_2^{-\frac{q-1}{d}}),
 \end{split}
 \end{align}
 where $m,n$ is a fixed solution to
 \begin{align}\label{ChainBezout}
 mp+n(q-1)=d.
 \end{align}
The image of the injective homomorphism 
 \begin{align*}
 \phi:\C^*&\rightarrow \Gamma_{\w}\\
 t&\mapsto (t^{\frac{q-1}{d}},t^{\frac{p}{d}}) 
 \end{align*}
 is an index $d$ subgroup of $\Gamma_\w$, but again we will only be interested in the maximal symmetry group. A semi-universal unfolding is given by
 \begin{align}\label{ChainUnfolding}
 \widetilde{\w}(x,y)=x^py+y^q+\sum_{\substack{0\leq i\leq p-2\\
 		0\leq j\leq q-1}}u_{ij}x^iy^j+u_{p-1,0}x^{p-1}.
 \end{align}
 By definition, $U_+$ is the subspace of $U$ containing elements such that there exists a positive integer $w_{ij}$ such that 
 \begin{align*}
 (t_1^pt_2)^{w_{ij}-1}=t_1^{w_{ij}-i}t_2^{w_{ij}-j}.
 \end{align*}
 For chain polynomials, there are five different cases of $U_+$ to consider:\\ 
 
 {Case I:} When $p,q>2$, the only solution is $i=j=w_{ij}=1$, and so $U_+=\Spec\C[u_{1,1}]=\A^1$. \\
 
 {Case II:} In the case where $p=2,\ q>2$ the only solution is $i=0$, $j=1$, $w_{ij}=2$, and so $U_+=\Spec\C[u_{0,1}]=\A^1$. \\
 
 {Case III:} In the case where $q=2, p>3$, we have $i=j=w_{ij}=1$, as well as $j=0$, $i=2$, and $w_{ij}=2$, and so $U_+=\Spec\C[u_{1,1},u_{2,0}]=\A^2$. \\
 
 {Case IV:} When $p=3,q=2$, we have $i=j=w_{ij}=1,\ j=0, i=2,\ w_{ij}=2,$ and $i=j=0, w_{ij}=3$, so $U_+=\Spec\C[u_{0,0},u_{1,1}, u_{2,0}]=\A^3$. \\
 
 {Case V:} In the case when $p=q=2$, we have $j=0$, $i=1$, $w_{ij}=3$, as well as $i=j=0$, $w_{ij}=4$, and $i=0,\ j=1,$ and $w_{ij}=2$, and so $U_+=\Spec\C[u_{0,0},u_{1,0},u_{0,1}]=\A^3$. 
 
 \subsection{Unfoldings of Brieskorn--Pham polynomials}\label{UnfoldingsofBPPolys}
 In the case of a two variable Brieskorn--Pham polynomial $\w=x^p+y^q$, we have $\mu=(p-1)(q-1)$, and 
 \begin{align}\label{BPWeights}
 (d_1,d_2;h)=(\frac{q}{d},\frac{p}{d}; \frac{pq}{d}),
 \end{align}
 where $d:=\gcd(p,q)$. One has that 
 \begin{align}\label{Jac BP}
 \Jac_\w=\text{span}\{1,x,\dots,x^{p-2}\}\otimes\text{span}\{1,y,\dots,y^{q-2}\},
 \end{align}
 and
 \begin{align}
 \begin{split}\label{BPMaxSymGp}
 \Gamma_{\w}=\big\{(t_1,t_2)\in(\C^*)^2|\ t_1^p=t_2^q\big\}&\xrightarrow{\sim}\C^*\times{\mu}_d\\
 (t_1,t_2)&\mapsto (t_1^nt_2^m,t_1^{\frac{p}{d}}t_2^{-\frac{q}{d}}),
 \end{split}
 \end{align}
 where $m,n$ is a fixed solution to
 \begin{align}\label{BPBezout}
 mp+nq=d.
 \end{align}
The image of the injective homomorphism 
 \begin{align*}
 \phi:\C^*&\rightarrow \Gamma_{\w}\\
 t&\mapsto (t^{\frac{q}{d}},t^{\frac{p}{d}})
 \end{align*}
 is an index $d$ subgroup of $\Gamma_\w$, but as in the loop and chain cases, we are only interested in the maximal symmetry group. A semi-universal unfolding is given by
 \begin{align}\label{BPUnfolding}
 \widetilde{\w}(x,y)=x^p+y^q+\sum_{\substack{0\leq i\leq p-2\\
 		0\leq j\leq q-2}}u_{ij}x^iy^j.
 \end{align}
 By definition, $U_+$ is the subspace of $U$ containing elements such that there exists a positive integer $w_{ij}$ such that 
 \begin{align*}
 (t_1^p)^{w_{ij}-1}=t_1^{w_{ij}-i}t_2^{w_{ij}-j}.
 \end{align*}
 For Brieskorn--Pham polynomials, we have the following five cases:\\
 
 {Case I:} In the case $p\geq q>3$, the only solution is $i=j=w_{ij}=1$, and so $U_+=\Spec\C[u_{1,1}]=\A^1$. \\
 
 {Case II:} In the case where $p=3$ and $q=2$, we have $i=1$, $j=0$ and $w_{ij}=4$, as well as $i=j=0$ and $w_{ij}=6$, so $U_+=\Spec\C[u_{0,0}, u_{1,0}]=\A^2$.\\
 
 {Case III:} In the case when $p=q=3$, we have $i=j=w_{ij}=1$, as well as $i=j=0$, $w_{ij}=3$, and so $U_+=\Spec \C[u_{0,0},u_{1,1}]=\A^2$. \\
 
 {Case IV:} In the case where $p=4$, $q=2$, we have $j=0$, $i=2$, $w_{ij}=2$, and $i=j=0$, $w_{ij}=4$. Therefore $U_+=\Spec \C[u_{0,0},u_{2,0}]=\A^2$. \\
 
 {Case V:} In the case where $p>4$ and $q=2$, we have $i=w_{ij}=2$ and $j=0$, so $U_+=\Spec\C [u_{2,0}]=\A^1$.

\section{Symplectic topology of the Milnor fibre}
Let $\Sigma$ be a smooth, compact, orientated surface of genus $g>0$ with $b>0$ connected boundary components $\partial\Sigma=\sqcup_{i=1}^{b}\partial_i\Sigma$. The surface to have in mind is the Milnor fibre of an invertible polynomial, $\check{V}_{\check{\w}}$. Note that by an abuse of notation, we will not distinguish between the Milnor fibre and its completion, since what we mean will be clear from context. 
\subsection{Graded symplectomorphisms}\label{graded symplecto section} In this subsection, we recall some facts about graded symplectic surfaces with the goal of providing a self-contained summary of Lemma \ref{graded homeo}. This provides criteria to ascertain when two graded symplectic surfaces are graded symplectomorphic, and is the key step in establishing Corollary \ref{HMS corollary}.\\

For a $2n$-dimensional symplectic manifold, $(X,\omega)$, there is a natural Lagrangian Grassmannian bundle $\text{LGr}(TX)\rightarrow X$, whose fibre at $x\in X$ is the Grassmannian of Lagrangian $n-$planes in $T_xX$.   Recall (\cite{SeidelBook}, \cite{BSMF_2000__128_1_103_0}) that we say $(X,\omega)$ is $\Z$-gradeable if it admits a lift to $\widetilde{\text{LGr}}(TX)$, the fibrewise universal cover of the Lagrangian Grassmannian bundle. This is possible if and only if $2c_1(X)=0$ in $H^2(X)$, and this implies that $K_X^{\otimes2}$, the square of the canonical bundle, is trivial. If $X$ is gradeable, then a grading is given by a choice of homotopy class of trivialisation of $K_X^{\otimes2}$. For a trivialising section $\Theta\in \Gamma(X,K_X^{\otimes2})$, one has a map 
\begin{align*}
\alpha_X: \text{LGr}(TX)&\rightarrow S^1\\
L_x\mapsto &\arg(\Theta|_{L_x}).
\end{align*}
Given a compact, exact Lagrangian submanifold, $L$, this defines a section of $\text{LGr}(TX)$ by considering the tangent space to $L$ at each point. We say that $L$ is gradeable with respect to a grading on $X$ if there exists a function $\alpha_X^\#:L\rightarrow \R$ such that $\exp(2\pi i\alpha_X^\#(x))=\alpha_X(T_xL)$. This is possible if and only if the Maslov class of $L$ vanishes, where the Maslov class is defined by the homotopy class of the map $L\rightarrow \text{LGr}(TX)\xrightarrow{\alpha_X} S^1$.\\

As explained in \cite[Section 13(c)]{SeidelBook}, on a (real) 2-dimensional surface, $\Sigma$, gradings correspond to trivialisations of the real projectivised tangent bundle, $\P_\R(T\Sigma)\simeq \text{LGr}(T\Sigma)$. Recall that a line field is a section of $\P_\R(T\Sigma)$. Supposing that a grading of $\Sigma$ is chosen such that $\alpha_\Sigma$ is as above, then one can define a line field on the surface given by $\eta=\alpha_\Sigma^{-1}(1)$. Conversely, a nowhere vanishing line field gives rise to a map $\alpha_\Sigma$ by recording the anticlockwise angle between the line field and any other line in the tangent plane. In this way, line fields correspond naturally to gradings on a surface, $\Sigma$.\\

Given a line field, $\eta$, which grades $\Sigma$, and a Lagrangian, $L$, represented by an embedded curve $\gamma:S^1\rightarrow \Sigma$, the map which corresponds to the Maslov class is given by recording the anticlockwise angle from $\eta_x$ to $T_xL$ at each point $x\in L$. The Maslov class vanishes, and hence $L$ is gradeable with respect to $\eta$, if and only if the sections $\gamma^*\eta$ and $\gamma^*TL$ are homotopic in $\gamma^*\P_\R(T\Sigma)$. A grading of $L$ is a choice of homotopy  between them.\\

We denote the space of line fields by $G(\Sigma):=\pi_0(\Gamma(\Sigma,\P_\R(T\Sigma)))$, and this has the natural structure of a torsor over the group of homotopy classes of maps $\Sigma\rightarrow S^1$, which we identify with $H^1(\Sigma)$. With this in mind, consider the trivial circle fibration 
\begin{align}
S^1\xrightarrow{\iota}\mathbb{P}_\R(T\Sigma)\xrightarrow{p}\Sigma,
\end{align}
which induces the exact sequence 
\begin{align}
0\rightarrow H^1(\Sigma)\xrightarrow{p^*}H^1(\mathbb{P}_\R(T\Sigma))\xrightarrow{\iota^*}H^1(S^1)\rightarrow 0.
\end{align}
Note that the orientation of $\Sigma$ induces an orientation on each tangent fibre, and so the map $\iota$ is unique up to homotopy. For each line field, we can associate an element $[\eta]\in H^1(\mathbb{P}_\R(T\Sigma))$ by considering the Poincar\'e--Lefschetz dual of $[\eta(\Sigma)]\in H_2(\mathbb{P}_\R(T\Sigma),\partial\mathbb{P}_\R(T\Sigma))$. These are precisely the elements such that $\iota^*([\eta])([S^1])=1$, and this is the content of \cite[Lemma 1.1.2]{Lekili2019}.\\

As already mentioned, for an embedded curve $\gamma:S^1\rightarrow\Sigma$, there is a corresponding section of the Lagrangian Grassmannian, $\tilde{\gamma}:S^1\rightarrow \P_\R(T\Sigma)$. This is given by $(\gamma,[T\gamma])$, where $[T\gamma]$ is the projectivisation of the tangent space to the curve $\gamma$.
\begin{definition}
	Given a line field, $\eta$, on $\Sigma$, and an immersed curve $\gamma:S^1\rightarrow \Sigma$, we define the winding number of $\gamma$ with respect to $\eta$ as
	\begin{align}\label{winding number def}
	w_\eta(\gamma):=\langle[\eta],[\tilde{\gamma}]\rangle,
	\end{align}
	where $\langle\cdot,\cdot\rangle:H^1(\P_\R(T\Sigma))\times H_1(\P_\R(T\Sigma))\rightarrow \Z$ is the natural pairing. 
\end{definition}
This pairing only depends on the homotopy class of $\eta$, as well as the regular homotopy class of $\gamma$. Recall that, for the case of surfaces, the Maslov number of a Lagrangian is precisely its winding number with respect to the line field used to grade the surface. Therefore, a Lagrangian is gradeable with respect to a line field if and only if its winding number with respect to this line field vanishes. Since we will be considering the Milnor fibre of a Lefschetz fibration, we must consider the grading on the Milnor fibre which is induced by the restriction of the unique grading of $\C^2$ to $\Sigma$. This is crucial so that the functor \eqref{FSRestriction} is graded, and therefore that \eqref{TEA VS} holds. The Lagrangian thimbles are contractible, and therefore gradeable, so each vanishing cycle is also gradeable with respect to the grading on the Milnor fibre induced from the restriction of the grading of $\C^2$. With this, we have that the grading on the Milnor fibre is given by a line field $\ell$ such that $w_\ell(\gamma_i)=0$ for each vanishing cycle $\gamma_i:S^1\rightarrow \Sigma$. Since the vanishing cycles form a basis of $H_1(\Sigma)$, the fact that the winding number around each Lagrangian is zero implies that the homotopy class of $\ell$ is unique.  \\

For any symplectomorphism $\phi:\Sigma_1\rightarrow\Sigma_2$ and $\eta_2\in G(\Sigma_2)$, one can consider the line field on $\Sigma_1$ given by
\begin{align}\label{diffeo action}
\phi^*(\eta_2)(x):=\big[(T_x\phi)^{-1}(\eta_2\circ\phi(x))\big]\quad\text{for all }x\in \Sigma_1.
\end{align}
If one has $(\Sigma_1;\eta_1)$ and $(\Sigma_2;\eta_2)$, where $\eta_1$ and $\eta_2$ are line fields used to grade the surfaces $\Sigma_1$ and $\Sigma_2$, respectively, we say that a symplectomorphism $\phi:\Sigma_1\rightarrow\Sigma_2$ is \emph{graded} if $\phi^*\eta_2$ is homotopic to $\eta_1$. If one takes $\Sigma_1=\Sigma_2$, then we define $\text{Symp}(\Sigma;\partial\Sigma)$ to be the space of symplectomorphisms of $\Sigma$ which fix $\partial\Sigma$ pointwise. One can then define the \emph{pure symplectic mapping class group} of $\Sigma$ as 
\begin{align}
\mathcal{M}(\Sigma;\partial\Sigma):=\pi_0(\text{Symp}(\Sigma;\partial\Sigma)),
\end{align}
and observe that this group acts on $G(\Sigma)$ as in \eqref{diffeo action}. The decomposition of $G(\Sigma)$ into $\mcg$-orbits is given in \cite[Theorem 1.2.4]{Lekili2019}, and this allows one to deduce \cite[Corollary 1.2.6]{Lekili2019}, which appears as Lemma \ref{graded homeo}, below. In what follows we briefly recall the relevant invariants, as well as techniques for their computation, in order to be able to state, and later utilise, Lemma \ref{graded homeo}. \\

For a given line field $\eta$, consider
\begin{align*}
w_\eta(\partial_i\Sigma),\qquad\text{for }i\in\{1,\dots,b\},
\end{align*}
the winding numbers around the boundary components. For two line fields to be homotopic, it is necessary for the winding numbers around each boundary component to agree, although this is definitely not sufficient. In particular, one can have two line fields which agree on the boundary, but which differ along interior non-separating curves. \\

Recall that for a closed, orientated Riemann surface, $\overline{\Sigma}$, a theorem of Atiyah in \cite{Atiyah} proves the existence of a quadratic form $\varphi:\mathcal{S}(\overline{\Sigma})\rightarrow\Z_2$, where $\mathcal{S}(\overline{\Sigma})$ is the space of spin structures on $\overline{\Sigma}$, $\varphi$ does not depend on the complex structure of $\overline{\Sigma}$, and the associated bilinear form on $H^1(\overline{\Sigma};\Z_2)$ is the cup product. Note that $\mathcal{S}(\overline{\Sigma})$ is a torsor over $H^1(\overline{\Sigma};\Z_2)$, and $\varphi$ being a quadratic form on $\mathcal{S}(\overline{\Sigma})$ means that it is a quadratic form on $H^1(\overline{\Sigma};\Z_2)$ for any choice of basepoint. Moreover, the associated bilinear form doesn't depend on the basepoint. He also proves that there are precisely two orbits of the mapping class group of $\overline{\Sigma}$ on $\mathcal{S}(\overline{\Sigma})$, and these are distinguished by the invariant $\varphi$, which is known as the \emph{Atiyah invariant}. In \cite{JohnsonSpinstructures}, Johnson gives a topological interpretation of the Atiyah invariant by proving that it is the Arf invariant of the corresponding quadratic form on $H_1(\overline{\Sigma},\Z_2)$. \\

The Arf invariant is well studied in topology, and we briefly recount some basic facts about it, as well as some computation techniques. Let $(\overline{V},(-\cdot-))$ be a vector space over $\Z_2$ with a non-degenerate bilinear form, and $\overline{q}:\overline{V}\rightarrow\Z_2$  a quadratic form satisfying
\begin{align}\label{reduced quad form condn}
\overline{q}(a+b)=\overline{q}(a)+\overline{q}(b)+(a\cdot b).
\end{align}
It is well-known that the Gau\ss\ sum 
\begin{align}
\text{GS}(\overline{q})=\sum_{x\in \overline{V}}(-1)^{\overline{q}(x)}=\pm 2^{\frac{\text{dim}\overline{V}}{2}},
\end{align}
and the sign is the \emph{Arf invariant} of the quadratic form. I.e. 
\begin{align}
\text{GS}(\overline{q})=(-1)^{\text{Arf}(\overline{q})} 2^{\frac{\text{dim}\overline{V}}{2}},
\end{align}
$\text{Arf}(\overline{q})\in\Z_2$.\\

To compute the Arf invariant, one can just compute the Gau\ss\  sum, although, except in particularly nice circumstances, this can become computationally intractable quite quickly. One can also find a base change to a symplectic basis where the formula simplifies, although we will not do this. Instead, consider the basis $\{e_1,\dots,e_{2n}\}$ of $\overline{V}$, and the matrix defined by 
\begin{align*}
f_{ii}&=\begin{cases}
2\qquad\text{if}\quad\overline{q}(e_{i})= 1\\
0\qquad\text{if}\quad\overline{q}(e_{i})=0
\end{cases}\\
f_{ij}&=\begin{cases}
1\qquad\text{if}\quad e_i\cdot{e_j}= 1\\
0\qquad\text{if}\quad e_i\cdot e_j=0
\end{cases}
\end{align*}
where $i\neq j$. Such a matrix defines an even quadratic form on a $\Z_{(2)}$ module, $V$, whose $\bmod\ 2$ reduction gives the bilinear pairing on $\overline{V}$. The precise module structure of $V$ is not important, since $\det f$ is well defined $\bmod\ 8$, and this value only depends on $\overline{q}$. One then has
\begin{align*}
\text{Arf}(\overline{q})=\begin{cases}
0\quad\text{if }\det f=\pm 1\bmod 8\\
1\quad\text{if }\det f=\pm 3\bmod 8.
\end{cases}
\end{align*}
The standard reference for further discussion of these facts is \cite[Chapter 9]{hirzebruch20060}.\\

Returning to the case at hand, recall that a non-vanishing vector field induces a spin structure on any compact Riemann surface with boundary. If the winding number around each boundary component with respect to this vector field is $2\bmod 4$, then this spin structure extends to the closed Riemann surface obtained by capping off the boundary components with discs, $\overline{\Sigma}$. Any vector field also yields a line field by considering the projectivisation, and each embedded curve has an even winding number with respect to this line field. Conversely, it is shown in \cite[Lemma 1.1.4]{Lekili2019} that if each embedded curve has even winding number with respect to a line field, then this line field arises as the projectivisation of a vector field. In light of this, in the case when two line fields have matching winding numbers around boundary components, arise from the projectivisation of vector fields, and where these vector fields define spin structures which extend to $\overline{\Sigma}$, one must check that the corresponding Atiyah invariants of these spin structures agree. \\

A useful fact is that, by the Poincar\'e--Hopf index theorem, (see, for example, \cite[Chapter 3]{HopfHeinz1983DGit}) for any compact $S\subseteq\Sigma$, we have
\begin{align}\label{winding number EC}
\sum_{i}^bw_\eta(\partial_i(S))=2\chi(S),
\end{align}
where $\chi(S)$ is the Euler characteristic. It is therefore clear that the winding number does not descend to a homomorphism from $H_1(\Sigma)$. What is true, however, is that one can consider for each line field $\eta$ the following homomorphism, given by the $\bmod\ 2$ reduction of the winding number:
\begin{align*}
[w_\eta]^{(2)}:H_1(\Sigma;\Z_2)\rightarrow\Z_2.
\end{align*}
From this, we can define the following invariant. 
\begin{definition}\label{sigma invariant defn}
	We define the $\Z_2$-valued invariant 
	\begin{align*}
	\sigma:G(\Sigma)&\rightarrow\Z_2\\
	\eta&\mapsto\begin{cases}
	0\qquad\text{if }[w_\eta]^{(2)}=0\\
	1\qquad\text{otherwise}.
	\end{cases}
	\end{align*}
\end{definition}

In the case when $\sigma(\eta)=0$, and so $\eta$ is the projectivisation of a vector field, $v$, we need to check when the spin structure on $\Sigma$ defined by $v$ extends to a spin structure on $\overline{\Sigma}$, and if it does, calculate the corresponding Atiyah invariant.\\

For a line field (not necessarily coming from the projectivisation of a vector field), $\eta$, the existence of a quadratic form
\begin{align*}
q_\eta:H_1(\Sigma;\Z_4)&\rightarrow\Z_4
\end{align*}
defined by 
\begin{align*}
q_\eta\big(\sum_{i=1}^m\alpha_i\big)=\sum_{i=1}^mw_\eta(\alpha_i)+2m\in \Z_4,
\end{align*}
where $\alpha_i$ are simple closed curves, and whose associated bilinear form is twice the intersection pairing on $H_1(\Sigma;\Z_4)$ is established in \cite[Proposition 1.2.2]{Lekili2019}. It is proven in \cite[Lemma 1.2.3]{Lekili2019} that for $g(\Sigma)\geq 2$, two line fields, $\eta,\ \theta$, lie in the same $\mcg$-orbit if the winding numbers agree on each boundary component, and $q_{\eta}=q_{\theta}$. In the case when $\eta$ and $\theta$ come from the projectivisation of vector fields, but the corresponding spin structures do not extend to $\overline{\Sigma}$, or when the two line fields do not arise as the projectivisation of vector fields, it is enough to show that $\sigma(\eta)=\sigma(\theta)$, and that the winding numbers on the boundary components agree. In the case where $\eta$ and $\theta$ are line fields such that $\sigma(\eta)=\sigma(\theta)=0$, and 
\begin{align}\label{EvenBoundaryWN's}
w_{\eta}(\partial_i(\Sigma))= w_{\theta}(\partial_i(\Sigma))\in 2+4\Z\quad\text{for each }i\in\{1,\dots,b\},
\end{align}
we must compare the corresponding Atiyah invariants. \\

Recall that the inclusion $\partial\Sigma\xhookrightarrow{i}\Sigma$ induces a map
\begin{align}
i_*:\Z_2^b\simeq H_1(\partial\Sigma;\Z_2)\rightarrow H_1(\Sigma;\Z_2)\simeq\Z_2^{2g+b-1}.
\end{align}
The kernel of the intersection pairing on $H_1(\Sigma;\Z_2)$ is spanned by the image of $i_*$, and the cokernel is naturally identified with $ H_1(\overline{\Sigma};\Z_2)$, where $\overline{\Sigma}$ is as above. The intersection form on $H_1({\Sigma};\Z_2)$ descends to a non-degenerate intersection form on $H_1(\overline{\Sigma};\Z_2)$.  \\

By the fact that $\sigma(\eta)=\sigma(\theta)=0$, we have that the function
\begin{align}
q/2:H_1(\Sigma;\Z_2)\rightarrow \Z_2
\end{align}
is well defined, where $q$ is either $q_{\eta}$ or $q_{\theta}$. By \eqref{EvenBoundaryWN's}, we have that $q/2(\partial_i\Sigma)\equiv 0\bmod 2$ for each $i\in \{1,\dots, b\}$. Since the kernel of the intersection pairing on $H_1(\Sigma;\Z_2)$ is spanned by the boundary curves,  $q/2$ descends to a non-singular quadratic form
\begin{align*}
\overline{q}:H_1(\overline{\Sigma};\Z_2)\rightarrow \Z_2
\end{align*}
such that
\begin{align}
\overline{q}(\alpha+\beta)=\overline{q}(\alpha)+\overline{q}(\beta)+(\alpha\cdot \beta),
\end{align}
and $\text{Arf}(\overline{q})$ gives the last invariant required to ascertain whether two line fields are in the same $\mcg$--orbit in the case where $g(\Sigma)\geq 2$. In the case when $g=1$, we define
\begin{align}\label{genus1invariant}
\tilde{A}(\eta):=\gcd\{w_\eta(\alpha),w_\eta(\beta),w_\eta(\partial_1\Sigma)+2,\dots,w_\eta(\partial_b\Sigma)+2\},
\end{align}
where $\alpha$ and $\beta$ are non-separating curves which project to a basis of $H_1(\Sigma;\Z_2)/\text{im}(i_*)$. \\

Putting this all together, \cite[Theorem 1.2.4]{Lekili2019} gives criteria for two line fields to be in the same mapping class group orbit. Using this, the authors give criteria for there to exist a graded symplectomorphism between two \emph{different} surfaces. 
\begin{lem}[{\cite[Corollary 1.2.6]{Lekili2019}}]\label{graded homeo}
Let $(\Sigma_1;\eta_1)$ and $(\Sigma_2;\eta_2)$ be two graded surfaces, each of genus $g$ with $b$ boundary components. There exists a symplectomorphism $\phi:\Sigma_1\rightarrow \Sigma_2$ such that $\phi^*(\eta_2)$ is homotopic to $\eta_1$ if and only if 
\begin{align*}
w_{\eta_1}(\partial_i\Sigma_1)&=w_{\eta_2}(\partial_i\Sigma_2),
\end{align*}
for each $i\in\{1,\dots,b\}$, and
\begin{itemize}
	\item If $g=1$, then $\tilde{A}(\eta_1)=\tilde{A}(\eta_2)$;
	\item If $g\geq 2$, then $\sigma(\eta_1)=\sigma(\eta_2)$ and, if the Arf invariant is defined, then  $\mathrm{Arf}(\overline{q}_{\eta_1})=\mathrm{Arf}(\overline{q}_{\eta_2})$.
\end{itemize}
\end{lem}

\subsection{Gluing cylinders}\label{gluing annuli}
In this subsection we describe a general construction of graded surfaces by gluing cylinders. This allows us to reduce the computation of topological invariants of these surfaces to the combinatorics of how they are glued. We then provide explicit descriptions of the Milnor fibres of invertible polynomials in two variables, as well as the corresponding computations of the topological invariants.\\

Let $A(\ell,r;m)$ denote $m$ disjoint cylinders placed in a column, each with $r$ marked points on the right boundary component, and $\ell$ marked points on the left. Considering each cylinder as a rectangle with top and bottom identified, for each $k\in\{1,\dots,m\}$, counting top-to-bottom in the column, we label the marked points on the right (resp. left) boundary component of the $k^{\text{th}}$ cylinder as $p_{r(k-1)}^+,\dots ,p_{rk-1}^+$ (resp. $p_{\ell(k-1)}^-,\dots, p_{\ell k-1}^-$). The reasoning for the labelling is that we would like to keep track of where the marked points are on each individual cylinder, as well as where each marked point is on the right (resp. left)  side of the column of cylinders with respect to the total ordering $p_0^+,\dots, p_{m_ir_i-1}^+$ (resp. $p_0^-,\dots, p_{m_i\ell_i-1}^-$). \\

Given a collection of cylinders 
\begin{align*}
A(\ell_1,r_1;m_1),A(\ell_2,r_2;m_2),\dots, A(\ell_n,r_n;m_n),
\end{align*}
such that $r_im_i=\ell_{i+1}m_{i+1}$, where $i$ is counted mod $n$, and corresponding permutations $\sigma_i\in \mathfrak{S}_{m_ir_i}$, we can glue these cylinders together in the following way. For each $i\in\{1,\dots n\}$ and $j\in\{0,\dots,m_ir_i-1\}$, we glue a small segment of the boundary component $p_j^+$ in $A(\ell_i,r_i;m_i)$ to $p_{\sigma_i(j)}^-$ in $A(\ell_{i+1},r_{i+1};m_{i+1})$ by attaching a strip.  See Figure \ref{fig:constructionexample} for an example.\\

\begin{figure}[h]
	\centering
	\includegraphics[width=0.8\linewidth]{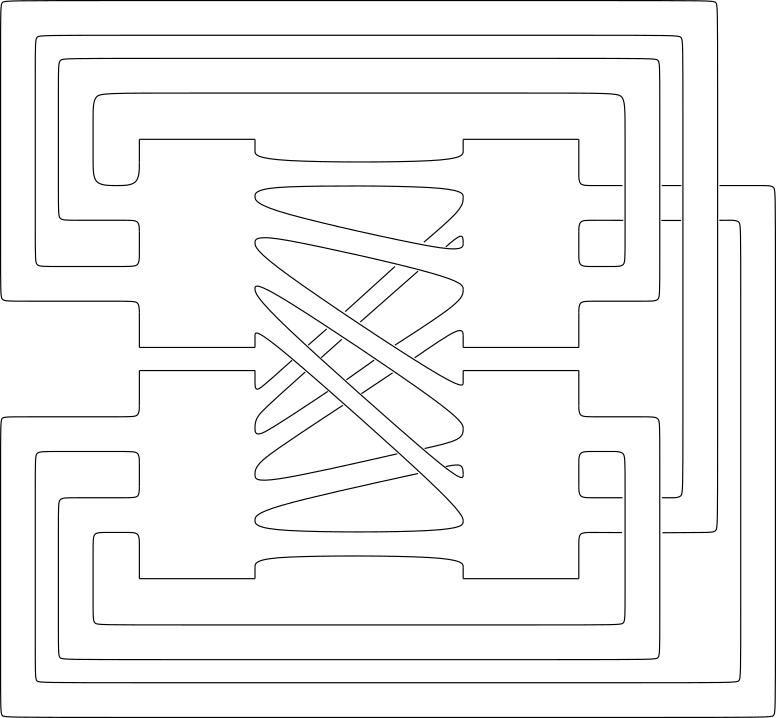}
	\caption{A genus 5 surface with 4 boundary components constructed by gluing $A(2,4; 2)$ to $A(4,2;2)$ via the permutations $\sigma_1=\protect\begin{psmallmatrix}
1&2&3&4&5&6&7&8\\
1&3&5&7&2&4&6&8
		\protect\end{psmallmatrix}$ and $\sigma_2=\protect\begin{psmallmatrix}
		1&2&3&4\\
		3&1&4&2
		\protect\end{psmallmatrix}$.}
	\label{fig:constructionexample}
\end{figure}
For each $i\in\{1,\dots, n\}$, the number of boundary components arising from gluing the $i^{\text{th}}$ and $(i+1)^{\text{st}}$ columns can be computed as follows. Consider the permutations
\begin{align*}
\tau_{r_i}=\big(0,r_i-1, r_i-2,\dots,1\big)\big(r_i,2r_i-1,2r_i-2,\dots, r_i+1\big)\dots\big((m_i-1)r_i,m_ir_i-1,\dots, (m_i-1)r_i+1\big)
\end{align*}
and 
\begin{align*}
\tau_{\ell_i}=\big(0,1,\dots,\ell_{i+1}-1\big)\big(\ell_{i+1},\dots, 2\ell_{i+1}-1\big)\dots\big((m_{i+1}-1)\ell_{i+1},\dots, m_{i+1}\ell_{i+1}-1\big).
\end{align*}
The number of boundary components between the $i^{\text{th}}$ and $(i+1)^{\text{st}}$ columns will then be given by the number of cycles in the decomposition of $\sigma_i^{-1}\tau_{\ell_{i+1}}\sigma_i\tau_{r_i}\in\mathfrak{S}_{m_ir_i}$. Note that if $m_i=m_{i+1}$ then we simply get the commutator. \\

To compute the homology groups of $\Sigma$, one can construct a ribbon graph
\begin{align}\label{graph}
\Gamma(\ell_1,\dots,\ell_n;r_1,\dots,r_n;m_1,\dots,m_n;\sigma_1,\dots,\sigma_n)\subseteq \Sigma,
\end{align}
onto which the surface deformation retracts. To do this, let there be a topological disc $\mathbb{D}^2$ for each of the cylinders. For each disc, attach a strip which has one end on the top, and the other end on the bottom. Then, attach a strip which connects two discs if there is a strip which connects the corresponding cylinders. These strips must be attached in such a way as to respect the cyclic ordering given by the gluing permutation. One can then deformation retract this onto a ribbon graph, whose cyclic ordering at the nodes is induced from the ordering of the strips on each cylinder. If there is no ambiguity, we will refer to this graph as $\Gamma(\Sigma)$. \\

Since the embedding of $\Gamma(\Sigma)$ into $\Sigma$ induces an isomorphism on homology, the homology groups of $\Sigma$ can be easily computed. Namely, since the graph is connected, we have $H_0(\Sigma)=\Z$. Since $\chi(\Sigma)=V-E=\text{rk}H_0(\Sigma)-\text{rk}H_1(\Sigma)=-\sum_{i=1}^nr_im_i$, we have $H_1(\Sigma)=\Z^{\oplus(1-\chi)}$. A basis for the first homology of the graph is given by an integral cycle basis, and so the basis of the first homology for $\Sigma$ is given by loops which retract onto these cycles. \\

Although there is no natural choice of grading on a surface glued in this way, in what follows we will only consider the case where the line field used to grade the surface is horizontal on each cylinder and parallel to the boundary components on attaching strips.
\subsubsection{Loop Polynomials}\label{gluing loop polys}
In the case of loop polynomials $\check{\w}=\check{x}^p\check{y}+\check{y}^q\check{x}$, we have that $n=3$ in the above construction, and we glue the cylinders
\begin{align*}
A(p-1,1;q-1),A(q-1,p-1;1), A(1,q-1;p-1),
\end{align*}
where $\sigma_1$ and $\sigma_2$ are the identity elements in $\mathfrak{S}_{q-1}$ and $\mathfrak{S}_{p-1}$, respectively, and $\sigma_3\in\mathfrak{S}_{(p-1)(q-1)}$ is given by \begin{align}
(q-1)(k_3-1)+i\mapsto (p-1)\big((-i)\bmod q-1\big)+(p-1-k_3),
\end{align}
where in this case $i\in\{0,\dots,q-2\}$ and $k_3\in\{1,\dots,p-1\}$. Call the resulting surface $\Sigma_\text{loop}(p,q)$. \\

For the basis of homology, we begin by considering the compact curves in each cylinder, $\gamma_i$. Together with these curves, we construct the basis for the first homology of the surface as follows. On each of the cylinders in the left and right columns, we take the curves to be approximately horizontal. We must therefore only describe the behaviour of the curves in the middle cylinder. Consider the curve which goes from the $((p-1)(k_1-1)+j)^{\text{th}}$ position on the left hand boundary to the $((q-1)(k_3-1)+i)^\text{th}$ position on the right hand boundary. In accordance with the construction of \cite[Section 3]{HabermannSmith}, this curve must wind $2\pi\Big(\frac{k_3-1}{p-1}+\frac{(1-k_1)\bmod  q-1}{q-1}\Big)$ degrees in the cylinder. This winding goes in the downwards direction, since we are thinking of the argument of the $\check{x}$ coordinate increasing in this direction. These curves form a basis of the first homology, since they retract onto a basis for the graph, $\Gamma(\Sigma_\text{loop}(p,q))$. The line field, $\ell$, used to grade the surface is approximately horizontal on each cylinder, and approximately parallel to the boundary on the connecting strips. By construction, we have $\sigma(\ell)=0$. See Figure \ref{milnorfibre} for the case of $\check{\w}=\check{x}^4\check{y}+\check{x}\check{y}^3$. 

\begin{figure}[h!]
	\centering
	\includegraphics[width=1\linewidth]{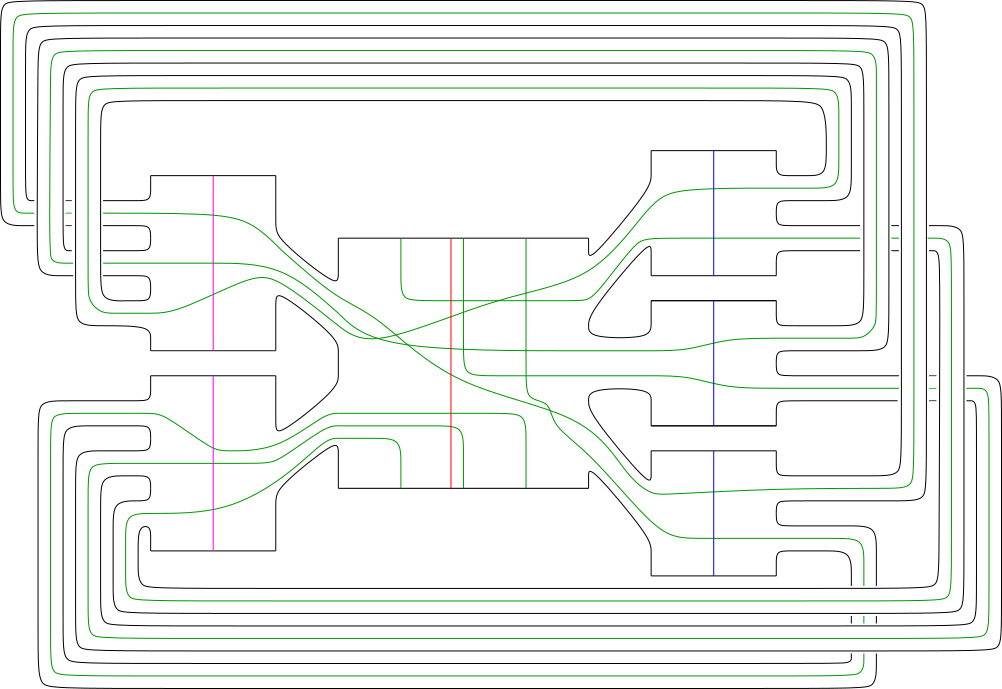}
	\caption{Milnor fibre for $\check{\w}=\check{x}^4\check{y}+\check{x}\check{y}^3$. Top and bottom of each cylinder are identified. Comparing with the basis of Lagrangians in \cite[Section 3]{HabermannSmith}, the red curve corresponds to $V_{\check{x}\check{y}}$, the purple ones to $\vc{\check{x}\check{w}}{i}$, the blue ones to $\vc{\check{y}\check{w}}{i}$, and the green ones to $\vc{0}{l,m}$. }
	\label{milnorfibre}
\end{figure}

There is only one boundary component between the first and second columns, as well as the second and third. With the line field $\ell$ given above, these components have winding numbers $-2(q-1)$ and $-2(p-1)$, respectively. To calculate the number of boundary components arising from gluing the third and first columns, note that in this case $\tau_{r_3}$ can be written as 
\begin{align}
(q-1)(k_3-1)+i\mapsto (q-1)(k_3-1)+\big((i-1)\bmod (q-1)\big),
\end{align}
and $\tau_{\ell_1}$ can be written as 
\begin{align}\label{left perm}
(p-1)(k_1-1)+j\mapsto (p-1)(k_1-1)+\big((j+1)\bmod(p-1)\big).
\end{align}
With this description, one can see that $\sigma_3^{-1}\tau_{\ell_1}\sigma_3\tau_{r_3}\in\mathfrak{S}_{(p-1)(q-1)}$ is given by
\begin{align}
(q-1)(k_3-1)+i\mapsto (q-1)\big((k_3-2)\bmod p-1\big)+\big((i-1)\bmod q-1\big).
\end{align}
As such, the length of a cycle is the least common multiple of $(p-1)$ and $(q-1)$, which is $\frac{(p-1)(q-1)}{\gcd(p-1,q-1)}$. There are therefore $\gcd(p-1,q-1)$ boundary components coming from gluing the third column to the first, each of winding number $-2\frac{(p-1)(q-1)}{\gcd(p-1,q-1)}$. We can then compute the genus from \eqref{winding number EC}, which yields
\begin{align*}
-2(p-1)-2(q-2)-2(p-1)(q-1)=2(2-2g_{\text{loop}}-\gcd(p-1,q-1)-2),
\end{align*}
and so the genus is 
\begin{align*}
g_{\text{loop}}=\frac{1}{2}(pq-1-\gcd(p-1,q-1)).
\end{align*}

By construction, the surface $\Sigma_{\text{loop}}(p,q)$ is graded symplectomorphic to the Milnor fibre of the polynomial $\check{\w}=\check{x}^p\check{y}+\check{x}\check{y}^q$. To see this, consider the ribbon graph which corresponds to the orientable surface $\check{V}_{\check{\w}}$. To construct this graph, first consider a disc $\mathbb{D}^2$ for each of the neck regions of the construction of the Milnor fibre in \cite[Section 3.1]{HabermannSmith}. Then, attach a thin strip which connects two discs if there is at least one vanishing cycle which goes between them. The cyclic ordering of the strips at each disc is determined by the ordering of the vanishing cycles passing through a corresponding neck region. This graph can then be embedded into $\check{V}_{\check{\w}}$ in such a way that all intersections occur on the interior of the discs, and away from the discs, the vanishing cycles are on the interior of the attaching strips. One can deformation retract this onto a graph with the induced cyclic ordering at the vertices. Call this graph $\Gamma(\check{V}_{\check{\w}})$, and observe that it is on-the-nose the same as $\Gamma(\Sigma_\text{loop}(p,q))$, and so the corresponding surfaces with boundary are symplectomorphic. See Figure \ref{Ribbon_graph} for an example of $p=4$, $q=3$. 

\begin{figure}[h!]
	\centering
	\begin{tikzpicture}
	\tikzset{every loop/.style={}}
	\node[circle,fill=black] at (0,0) {} edge [in=250,out=110,loop] ();
	\node[circle,fill=black] at (2,2) {} edge [in=340,out=200,loop] ();
	\node[circle,fill=black] at (4,2) {} edge [in=340,out=200,loop] ();
	\node[circle,fill=black] at (6,2) {} edge [in=340,out=200,loop] ();
	\node[circle,fill=black] at (-2,2) {} edge [in=340,out=200,loop] ();
	\node[circle,fill=black] at (-4,2) {} edge [in=340,out=200,loop] ();
	\draw (0,0) to [in=270,out=40] (2,2);
	\draw (0,0) to [in=270,out=150](-2,2);
	\draw (0,0) to [in=270,out=20] (4,2);
	\draw (0,0) to [in=270,out=0] (6,2);
	\draw (0,0) to [in=270,out=180](-4,2);
	\draw (-2,2) to [in=135,out=30](6,2);
	\draw (-2,2) to [in=135,out=90](4,2);
	\draw (-2,2) to [in=135,out=150](2,2);
	\draw (-4,2) to [in=45,out=30] (6,2);
	\draw (-4,2) to [in=45,out=90] (4,2);
	\draw (-4,2) to [in=45,out=150] (2,2);
	\end{tikzpicture}
	\caption{Ribbon graph for $\Gamma(\check{V}_{\check{\w}})= \Gamma(\Sigma_\text{loop}(4,3))$, where the cyclic ordering of the half-edges at the nodes is in the anticlockwise direction.}
	\label{Ribbon_graph}
\end{figure}
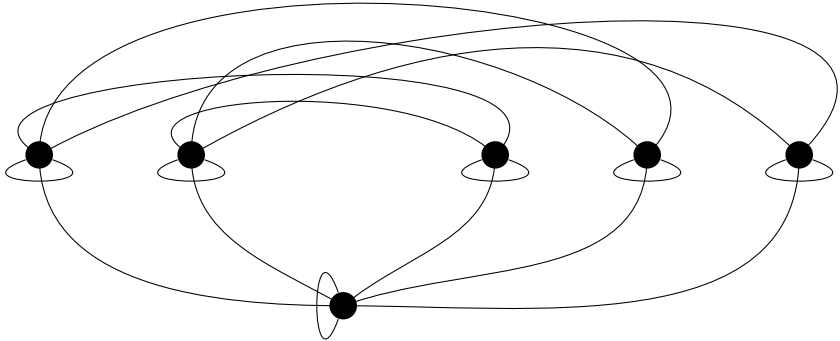

To see that the two surfaces are graded symplectomorphic, consider the corresponding fat graphs in both cases. In this situation one can see that the description of the line field used to grade $\Sigma_{\text{loop}}(p,q)$ agrees with the description of the line field used to grade $\check{V}_{\check{\w}}$, as in \cite[Section 3.7]{HabermannSmith}, and this shows that the surfaces are graded symplectomorphic. 

\subsubsection{Chain polynomials}\label{gluing chain polys}
In the case of chain polynomials, we have $\check{\w}=\check{x}^p+\check{x}\check{y}^q$, and we will show that the Milnor fibre can be constructed by gluing 
\begin{align*}
A(p-1,1;q-1),A(q-1,(p-1)(q-1);1),
\end{align*}
where $\sigma_1$ is the identity element in $\mathfrak{S}_{q-1}$, and $\sigma_2\in\mathfrak{S}_{(p-1)(q-1)}$ is given by 
\begin{align}
i \mapsto (p-1)(-i\bmod q-1)+p-2-\big\lfloor\frac{i}{q-1}\big\rfloor,
\end{align}
where in this case $i\in\{0,\dots,(p-1)(q-1)-1\}$. Call the resulting surface $\Sigma_\text{chain}(p,q)$. \\

For the basis of homology, we begin by including the compact curves in each cylinder, $\gamma_i$. Together with these curves, we construct a basis for homology as follows. On the cylinders in the first column, we take the curves to be approximately horizontal. In the cylinder in the second column, the curve going from the $(p-1)(k_2-1)+j)^\text{th}$ position on the left hand side to the $i^{\text{th}}$ position on the right hand side winds $2\pi\Big(\frac{(1-k_1)\bmod p-1}{p-1}+\frac{i}{(p-1)(q-1)}\Big)$ degrees, again in the downwards direction. This is in accordance with the description of the curves as in \cite[Section 5.2]{HabermannSmith}.  Together, these curves form a basis for the first homology, since they retract onto a basis of the corresponding ribbon graph, $\Gamma(\Sigma_{\text{chain}}(p,q))$. As in the loop case, the line field, $\ell$, used to grade the surface is approximately horizontal on each cylinder, and approximately parallel to the boundary on the connecting strips. By construction, we have $\sigma(\ell)=0$. \\

There is only one boundary component which arises from gluing the first and second columns, and the winding number around this boundary component is $-2(q-1)$. To compute the number of boundary components, and their winding numbers, arising from gluing the second column to the first, observe that in this case, $\tau_{r_2}$ is just the permutation $j\mapsto j-1$, and $\tau_{\ell_{1}}$ is of the same form as \eqref{left perm}. The permutation  $\sigma_2^{-1}\tau_{\ell_1}\sigma_2\tau_{r_2}\in\mathfrak{S}_{(p-1)(q-1)}$
is given by
\begin{align}
i\mapsto i-q.
\end{align}
Therefore the length of a cycle in the above permutation is $\frac{(p-1)(q-1)}{\gcd(p-1,q)}$. From this, we see that there are $\gcd(p-1,q)$ boundary components arising from this gluing, and each boundary component has winding number $-2\frac{(p-1)(q-1)}{\gcd(p-1,q)}$.  Therefore there are $1+\gcd(p-1,q)$ boundary components in total, and we conclude from \eqref{winding number EC} that
\begin{align*}
g_{\text{chain}}=\frac{1}{2}(pq-p+1-\gcd(p-1,q)).
\end{align*}

As in the loop case, we claim that the surface constructed above is graded symplectomorphic to $\check{V}_{\check{\w}}$. To see this, we can construct a ribbon graph corresponding to $\check{V}_{\check{\w}}$ as in the case of loop polynomials. This graph also matches $\Gamma(\Sigma_{\text{chain}}(p,q))$ on-the-nose, and this establishes that $\Sigma_{\text{chain}}(p,q)$ and $\check{V}_{\check{\w}}$ are symplectomorphic. To see that they are graded symplectomorphic, observe that in the corresponding fat graphs, the description of the line field above agrees with the description as in \cite[Section 5.3]{HabermannSmith}, and this shows that the surfaces are graded symplectomorphic. 
\subsubsection{Brieskorn--Pham polynomials}\label{gluing BP polys}
In the case of Brieskorn--Pham polynomials, we have\\ $\check{\w}=\check{x}^p+\check{y}^q$, where $(p,q)\neq(2,2)$. Consider the surface obtained by gluing one cylinder to itself with the permutation $\sigma\in\mathfrak{S}_{(p-1)(q-1)-1}$, which is given by
\begin{align}
i\mapsto -i(q-1),
\end{align}
where in this case $i$ is a point on the right boundary, and is considered as an element of \\ $\{0,\dots,(p-1)(q-1)-2\}$. Call this surface $\Sigma_{BP}(p,q)$.\\

For the basis of homology, we take a compact vertical curve in the cylinder, $\gamma_1$, as well as one curve which is approximately parallel to the boundary along each of the connecting strips. On the interior of the cylinder, we have that the curve beginning in the $j^\text{th}$ position on the left hand side and ending at the $i^\text{th}$ position on the right hand side must wind $2\pi\Big(\frac{i+(-j)\bmod [(p-1)(q-1)-1]}{(p-1)(q-1)-1}\Big)$ degrees in the downwards direction, in accordance with the description of the curves in \cite[Section 6.2]{HabermannSmith}. Together, these curves form a basis for the first homology of $\Sigma_{BP}(p,q)$, since they retract onto a basis of the corresponding ribbon graph, $\Gamma(\Sigma_{BP}(p,q))$. As in the previous two cases, the line field, $\ell$, used to grade the surface is approximately horizontal on the cylinder, and approximately parallel to the boundary on the connecting strips. Again, by construction, we have $\sigma(\ell)=0$.\\

Let $\tau$ be the permutation $i\mapsto i-1$, so the number of boundary components is given by the number of cycles in the decomposition $[\sigma,\tau]\in\mathfrak{S}_{(p-1)(q-1)-1}$. The commutator is given by
\begin{align*}
i\mapsto i-p, 
\end{align*}
and so the length of a cycle will be given by $\frac{(p-1)(q-1)-1}{\gcd(p,q)}$. There are therefore $\gcd(p,q)$ boundary components arising from this gluing, and each has winding number $-2\frac{(p-1)(q-1)-1}{\gcd(p,q)}$. Therefore, we have
\begin{align*}
g_{BP}=\frac{1}{2}((p-1)(q-1)+1-\gcd(p,q)).
\end{align*}
As in the previous cases, we deduce that that $\Sigma_{BP}(p,q)$ is graded symplectomorphic to the Milnor fibre. 
\subsection{Symplectic cohomology of the Milnor fibre}\label{SHMilnorFibre}
In this subsection we utilise the explicit descriptions of the Milnor fibres of invertible polynomials given above to calculate the module structure of symplectic cohomology of these surfaces. By combining this with Theorem \ref{HHSHisom} below, we will be able to deduce the correct mirror curves in the proof of Theorem \ref{main theorem}.\\

The symplectic cohomology of surfaces admits a particularly simple description -- namely, for any Riemann surface, $\Sigma_{g,b}$, of genus $g>0$ with $b>0$ boundary components, we have 
\begin{align}\label{symplectic cohomology}
\SH^*(\Sigma_{g,b})\simeq H^*(\Sigma_{g,b})\oplus\bigoplus_{i=1}^b\bigg(\bigoplus_{k\geq 1}H^*(S^1)[k\cdot w_\eta(\partial_i\Sigma_{g,b})]\bigg),
\end{align}
where $w_\eta(\partial_i\Sigma_{g,b})$ is the winding number of the line field $\eta$ about the boundary component $\partial_i\Sigma_{g,b}$. This was first described in the case of one puncture in \cite[Example 3.3]{seidel2008}, and the generalisation to more than one puncture follows by the same argument. Note that the grading convention in \cite{seidel2008} is shifted by one from ours.  \\

In the case of loop polynomials, $\check{\w}=\check{x}^p\check{y}+\check{y}^q\check{x}$, we saw in Section \ref{gluing loop polys} that the Milnor fibre is a $2+\gcd(p-1,q-1)$-times punctured surface of genus $g_\text{loop}=\frac{1}{2}(pq-1-\gcd(p-1,q-1))$. Consider $\Sigma_{g,b}=\check{V}_{\check{\w}}$, and let $\ell$ be the line field used to grade the surface as in Section \ref{gluing loop polys}. We then have by \eqref{symplectic cohomology} and the analysis in Section \ref{gluing loop polys}, that
\begin{align*}
\SH^0(\check{V}_{\check{\w}})&\simeq\C\\
\SH^1(\check{V}_{\check{\w}})&\simeq\C^{\oplus pq}\\
\SH^{2n(p-1)}(\check{V}_{\check{\w}})&\simeq\SH^{2n(p-1)+1}(\check{V}_{\check{\w}})\simeq\C\qquad\text{for }n\in\Z_{>0} \text{ such that }\frac{q-1}{\gcd(p-1,q-1)}\nmid n\\
\SH^{2n(q-1)}(\check{V}_{\check{\w}})&\simeq \SH^{2n(q-1)+1}(\check{V}_{\check{\w}})\simeq\C\qquad\text{for }n\in\Z_{>0} \text{ such that }\frac{p-1}{\gcd(p-1,q-1)}\nmid n\\
\SH^{2n\frac{(p-1)(q-1)}{\gcd(p-1,q-1)}}(\check{V}_{\check{\w}})&\simeq \SH^{2n\frac{(p-1)(q-1)}{\gcd(p-1,q-1)}+1}(\check{V}_{\check{\w}})\simeq\C^{\oplus(2+\gcd(p-1,q-1))}\qquad\text{for }n\in\Z_{>0}.
\end{align*}

In the case of chain polynomials, $\check{\w}=\check{x}^p+\check{x}\check{y}^q$, we have that the Milnor fibre is a $(1+\gcd(p-1,q))-$times punctured surface of genus $g_\text{chain}=\frac{1}{2}(pq-p+1-\gcd(p-1,q))$. Let $\ell$ be the line field used to grade the surface, as in Section \ref{gluing chain polys}. We then have by \eqref{symplectic cohomology} and the analysis in Section \ref{gluing chain polys} that
\begin{align*}
\SH^0(\check{V}_{\check{\w}})&\simeq\C\\
\SH^1(\check{V}_{\check{\w}})&\simeq\C^{\oplus pq-p+1}\\
\SH^{2n(q-1)}(\check{V}_{\check{\w}})&\simeq\SH^{2n(q-1)+1}(\check{V}_{\check{\w}})\simeq\C\qquad\text{for }n\in\Z_{>0} \text{ such that }\frac{p-1}{\gcd(p-1,q)}\nmid n\\
\SH^{2n\frac{(p-1)(q-1)}{\gcd(p-1,q)}}(\check{V}_{\check{\w}})&\simeq\SH^{2n\frac{(p-1)(q-1)}{\gcd(p-1,q)}+1}(\check{V}_{\check{\w}})\simeq\C^{\oplus(1+\gcd(p-1,q))}\qquad\text{for }n\in\Z_{>0}.
\end{align*}
In the case of Brieskorn--Pham polynomials, we have that the Milnor fibre is a $\gcd(p,q)$-times punctured surface of genus $g_{BP}=\frac{1}{2}((p-1)(q-1)+1-\gcd(p,q))$. Let $\ell$ be the line field used to grade the surface, as in Section \ref{gluing BP polys}. Then, by \eqref{symplectic cohomology} and the analysis in Section \ref{gluing BP polys}, we have
\begin{align*}
\SH^0(\check{V}_{\check{\w}})&\simeq\C\\
\SH^1(\check{V}_{\check{\w}})&\simeq\C^{\oplus (p-1)(q-1)}\\
\SH^{2n\frac{(p-1)(q-1)-1}{\gcd(p,q)}}(\check{V}_{\check{\w}})&\simeq\SH^{2n\frac{(p-1)(q-1)-1}{\gcd(p,q)}+1}(\check{V}_{\check{\w}})\simeq \C^{\oplus\gcd(p,q)}\qquad\text{for }n\in\Z_{>0}.
\end{align*}
As previously mentioned, the comparison of the symplectic cohomology of the Milnor fibre and the Hochschild cohomology of the Fukaya category of the Milnor fibre will be crucial in our mirror symmetry argument. To this end, we have the following theorem of Lekili and Ueda:
\begin{thm}[{\cite[Corollary 6.6]{2018arXiv180604345L}}] \label{HHSHisom}
	Let $\check{\w}$ be the transpose of an invertible polynomial in two variables such that $\check{d}_{0}>0$. Then 
	\begin{align*}
	\mathrm{SH}^*(\check{V}_{\check{\w}})\simeq\mathrm{HH}^*(\mathcal{F}(\check{V}_{\check{\w}}))
	\end{align*}
\end{thm}
Note that assuming $\check{d}_0>0$ is crucial, as can be seen if one considers $\check{\w}=\check{x}^2+\check{y}^2$.
\subsection{Graded symplectomorphisms between Milnor fibres}\label{graded homeos between milnor fibres}
It is a natural question to ask which Milnor fibres are graded symplectomorphic, and in this subsection we utilise Lemma \ref{graded homeo} to determine this. Since the genera, number of boundary components, and winding numbers around boundary components of the Milnor fibres were calculated above, it is easy to check when these match. This gives the potential graded symplectomorphisms, although one must also check that the corresponding Arf invariants agree whenever they are defined. We use the method described in Section \ref{graded symplecto section} to compute the Arf invariant when necessary. By establishing which Milnor fibres are graded symplectomorphic, Corollary \ref{HMS corollary} follows immediately once Theorem \ref{main theorem} is proven.\\

Observe that for each $q\geq 2$ and $n\geq 1$, we have that  $\check{\w}_{\text{loop}}=\check{x}^{(q-1)n+1}\check{y}+\check{y}^q\check{x}$, and $\check{\w}_{\text{chain}}=\check{x}^{qn+1}+\check{y}^q\check{x}$ have the same genus, number of boundary components, and winding numbers along each boundary component. In the case of $q$ odd, this is enough to give a graded symplectomorphism by Lemma \ref{graded homeo}, since $\sigma=0$ in both cases, and $-2(q-1)\equiv 0\bmod 4$. In the case where $q$ and $n$ are both even, we again have that the Milnor fibres are graded symplectomorphic. In the case where $q$  is even and $n$ is odd, it remains to check that the relevant Arf invariants agree. \\

For a graded symplectomorphism between the Milnor fibres of a chain and Brieskorn--Pham polynomial, we have that $\check{\w}'_{\text{chain}}=\check{x}^{p}+\check{y}^{n(p-1)}\check{x}$ and $\check{\w}_{\text{BP}}=\check{x}^{p}+\check{y}^{np}$ for each $p\geq 2$ and $n\geq 1$ have the same genus, number of boundary components, and winding numbers along each boundary component. In the case where $n$ is even and $p$ is odd, we have that $-2(n(p-1)-1)\equiv0\bmod 4$, and so Lemma \ref{graded homeo} gives us a graded symplectomorphism between the Milnor fibres. Similarly, for $p=2$ and $n$ odd, Lemma \ref{graded homeo} yields a graded symplectomorphism between Milnor fibres. In all other cases, we must check the relevant Arf invariants. \\

The only possibility for a graded symplectomorphism between the Milnor fibres of a loop and Brieskorn--Pham polynomial is that both are symplectomorphic to a Milnor fibre of a chain polynomial. For such a graded symplectomorphism to exist, we require $\check{\w}_{\text{loop}}=\check{x}^{q}\check{y}+\check{y}^q\check{x}$,  $\check{\w}_{\text{chain}}=\check{x}^{q+1}+\check{y}^q\check{x}$, and $\check{\w}_{\text{BP}}=\check{x}^{q+1}+\check{y}^{q+1}$.
It should be noted that the potential graded symplectomorphisms discussed above are the only such possibilities. 
\subsubsection{Graded symplectomorphisms between the Milnor fibres of loop and chain polynomials}
 In the case of loop polynomials of the form $\check{\w}_{\text{loop}}=\check{x}^{(q-1)n+1}\check{y}+\check{y}^q\check{x}$, we have that there are $q+1$ boundary components. Recall the basis of the first homology of the Milnor fibre given in \cite[Section 3]{HabermannSmith}.  An elementary calculation shows that if we remove the Lagrangian $V_{\check{x}\check{y}}^{\text{loop}}$, as well as the Lagrangians $\{\vcc{\check{x}\check{w}}{i}{\text{loop}}\}_{i\in\{0,\dots,q-2\}}$, then the restriction of the intersection form is non-degenerate. \\

In the case of chain polynomials of the form $\check{\w}_{\text{chain}}=\check{x}^{qn+1}+\check{y}^q\check{x}$, we consider the basis of Lagrangians for the first homology group of the Milnor fibre as given in \cite[Section 5]{HabermannSmith}. By removing the Lagrangian $V_{\check{x}\check{y}}^{\text{chain}}$, as well as the Lagrangians $\{\vcc{\check{x}\check{w}}{i}{\text{chain}}\}_{i\in\{0,\dots,q-2\}}$, the restriction of the intersection form to the remaining Lagrangians is non-degenerate. \\

Let $U_n$ be the $n\times n$ matrix given by $(U_n)_{i,j}=\begin{cases}
1\quad\text{if }i\geq j\\
0\quad\text{otherwise}
\end{cases}$. 
Then we have that $f_{\text{chain}}=U_{q-1}\otimes U_{qn}+(U_{q-1}\otimes U_{qn})^T$. On the other hand, $f_{\text{loop}}$ is the block matrix given by 
\begin{align*}
\begin{pmatrix}
2\text{Id}_{n(q-1)}& \text{Id}_{n(q-1)}&\dots &\text{Id}_{n(q-1)}\\
\text{Id}_{n(q-1)} \\
\vdots& &U_{q-1}\otimes U_{n(q-1)}+(U_{q-1}\otimes U_{n(q-1)})^T\\
\text{Id}_{n(q-1)}
\end{pmatrix}
\end{align*}
In both cases, one can explicitly compute that the determinant is $nq+1$, and so, in particular, we have $\text{Arf}(\overline{q}_{\text{chain}})=\text{Arf}(\overline{q}_{\text{loop}})$. Therefore, by Lemma \ref{graded homeo}, the surfaces are graded symplectomorphic, and their respective compact Fukaya categories are quasi-equivalent. 
\subsubsection{Graded symplectomorphisms between the Milnor fibres of chain and Brieskorn--Pham polynomials} In the case of chain polynomials of the form $\check{\w}'_{\text{chain}}=\check{x}^{p}+\check{y}^{n(p-1)}\check{x}$, and Brieskorn--Pham polynomials of the form $\check{\w}_{\text{BP}}=\check{x}^p+\check{y}^{np}$, we have that the there are $p$ boundary components. In the chain case, we remove $V_{\check{x}\check{y}}^{\text{chain}}$, as well as the Lagrangians $\{\vcc{\check{x}\check{w}}{i}{\text{chain}}\}_{i\in\{0,\dots,p-3\}}$ from the collection of Lagrangians which form a basis of the first homology of the Milnor fibre, and the restriction of the intersection form to the remaining Lagrangians is non-degenerate. In the Brieskorn--Pham case, if we remove the Lagrangians $\{\vcc{0}{l,np-2}{\text{BP}}\}_{l\in\{0,\dots,p-2\}}$ from the collection of Lagrangians which form a basis of the first homology group of the Milnor fibre, as described in \cite[Section 6]{HabermannSmith}, then the restriction of the intersection form to the remaining Lagrangians is likewise non-degenerate. \\

In the case of chain polynomials, we have that $f_{\text{chain}'}$ is given by
removing the top and left $p-2$ rows and columns from
\begin{align*}
\begin{pmatrix}
2\text{Id}_{n(p-1)-1}& \text{Id}_{n(p-1)-1}&\dots &\text{Id}_{n(p-1)-1}\\
\text{Id}_{n(p-1)-1} \\
\vdots& &U_{p-1}\otimes U_{n(p-1)-1}+(U_{p-1}\otimes U_{n(p-1)-1})^T\\
\text{Id}_{n(p-1)-1}
\end{pmatrix}
\end{align*}
In the case of Brieskorn--Pham polynomials, we have that $f_{\text{BP}}=U_{p-1}\otimes U_{np-2}+(U_{p-1}\otimes U_{np-2})^T$.\\

In both cases, we have that 
\begin{align*}
\det f_{\text{chain}'}=\det f_{\text{BP}}=\begin{cases}
p\quad&\text{if p is odd}\\
np-1\quad&\text{if p is even}.
\end{cases}
\end{align*}
We therefore have by Lemma \ref{graded homeo} that the Milnor fibres are graded symplectomorphic. 
\section{Hochschild cohomology via matrix factorisations}\label{HHcomputation section}
In this section, we make the necessary Hochschild cohomology computations which will later enable us to deduce the existence of an affine scheme of finite type which represents the moduli functor of $A_\infty$-structures on the graded algebras we are interested in. This is the main computational component of the paper, and we include the entire calculation for completeness, although a computation of $\HH^n(Y)$ for $n\leq 2$ would have sufficed. \\

Suppose once more that we are in the setting of Section \ref{background section}, and we have that $\w$ is an invertible polynomial in two variables such that $d_0>0$, $\Gamma$ is a subgroup of $\Gamma_\w$ of finite index containing $\phi(\C^*)$, and $\W_u$ the quasi-homogenisation of a semi-universal unfolding corresponding to $u\in U_+$. Denote $V=\{x,y,z\}$, $S:=\text{Sym}\, V=\C[x,y,z]$, and so $R_u=S/(\W_u)$, and $\W_u\in (S\otimes\chi)^\Gamma$ (recall $\chi=\chi_{\w}|_{\Gamma}$). Equation \eqref{coh mf quasi equiv} implies that 
\begin{align}
\HH^*(Y_u)\simeq \HH^*(\A^3, \Gamma,\W_u).
\end{align}
This vastly simplifies the calculation at hand, since a theorem of Ballard, Favero, and Katzarkov (\cite[Theorem 1.2]{BFK}) reduces the computation of the Hochschild cohomology of the category of $\Gamma$--equivariant matrix factorisations of $\W_u$ to studying the cohomology of certain Koszul complexes, which in nice cases reduces to studying the Jacobi algebra of $\W_u$. To this end, consider an element $\gamma\in \ker\chi$, and $V_\gamma$ the subspace of $V$ of $\gamma$-invariant elements. Let $S_\gamma:=\text{Sym}\, V_\gamma$, and $N_\gamma$ the complement of $V_\gamma$ in $V$, so that $V\simeq V_\gamma\oplus N_\gamma$ as a $\Gamma$-module. Denote by $\W_\gamma$ the restriction of $\W_u$ to $\Spec S_\gamma$, and consider the Koszul complex
\begin{align}\label{Koszul cpx}
C^*(d\W_\gamma):=\{\dots\rightarrow \wedge^2V_\gamma^\vee\otimes\chi^{\otimes(-2)}\otimes S_\gamma\rightarrow V_\gamma^\vee\otimes\chi^\vee\otimes S_\gamma\rightarrow S_\gamma\},
\end{align}
where $S_\gamma$ sits 
in cohomological degree $0$, and the differential is the contraction with 
\begin{align}\label{Koszul differetial}
d\W_\gamma\in\big(V_\gamma\otimes\chi\otimes S_\gamma\big)^\Gamma.
\end{align}
Denote by $H^i(d\W_\gamma)$ the $i^{\text{th}}$ cohomology group of the Koszul complex. The zeroth cohomology of \eqref{Koszul cpx} is isomorphic
to the Jacobi algebra of $\W_\gamma$, and if $\W_\gamma$ 
has an isolated critical point at the origin, then $C^*(d\W_\gamma)$ is a resolution. Our main tool for computing Hochschild cohomology is the following theorem:

\begin{thm}[\cite{BFK}]\label{HH comp tool thm}
	Let $\w$ be an invertible polynomial in two variables, $\Gamma$ be a subgroup of $\Gamma_{\w}$ of finite index containing $\phi(\C^*)$ acting on  $\A^{3}=\Spec S$, and $\W_u\in S$ be a non-zero element of degree $\chi$. Assume that the singular locus of the zero set 
	$Z_{(-\W_u)\boxplus \W_u}$ of the 
	Thom--Sebastiani sum $-\W_u\boxplus \W_u$ is contained in the product of the zero sets $Z_{\W_u}\times Z_{\W_u}$. Then $\mathrm{HH}^t(\A^{3},\Gamma,\W_u)$ is isomorphic to 
	\begin{align}\label{HHComputation tool}
	\begin{split}
	\Bigg(\bigoplus_{\substack{\gamma\in\ker\chi,\ l\geq 0\\ t-\dim N_\gamma=2u}}H^{-2l}(d\W_\gamma)&\otimes \chi^{\otimes(u+l)}\otimes\wedge^{\dim N_\gamma}N_\gamma^\vee\\
	&\oplus \bigoplus_{\substack{\gamma\in\ker\chi,\ l\geq 0\\ t-\dim N_\gamma=2u+1}}H^{-2l-1}(d\W_\gamma)\otimes \chi^{\otimes(u+l+1)}\otimes\wedge^{\dim N_\gamma}N_\gamma^\vee\Bigg)^\Gamma.
	\end{split}
	\end{align}
\end{thm}

In the case where the $\Gamma$-action on $V$ satisfies $\dim (S\otimes\rho)^\Gamma<\infty$ for any $\rho\in\hat{\Gamma}$, one then has
\begin{align}\label{finite HH}
\dim \HH^t(\A^{3},\Gamma,\W)<\infty
\end{align}
for every $t\in \Z$. To see this, note that the complex $C^*(d\W_\gamma)$ is
always bounded, and the group $\ker\chi$ is finite. Therefore, 
each direct summand of \eqref{HHComputation tool} is finite dimensional, and 
there are only finitely many $u$ contributing to a fixed $t$. \\

Theorem \ref{HH comp tool thm} is a minor modification of \cite[Theorem 1.2]{BFK}, where the difference is in the convention for the Koszul complex. In our case, when there is an additional $\C^*$-action on $V$, then \eqref{HHComputation tool} is equivariant with respect to it. In particular, in the case of $u=0\in U_+$, we have that there is an additional $\C^*$-action on $V$ given by $t\cdot(x,y,z)=(x,y,tz)$, and this induces an additional $\C^*$-action on $\HH^*(Y_0)$. Denote by $\HH^*(Y_0)_{<0}$ the negative weight part of this action. We refer the reader to \cite{BFK} for a proof of Theorem \ref{HH comp tool thm}.
\begin{definition}\label{twisting}
	We will say that the pair $(\w,\Gamma)$ is \emph{untwisted} if $\mathrm{HH}^2(Y_0)_{<0}$ comes only from the summand $(\mathrm{Jac}_{\w}\otimes\C[z]\otimes\chi)^\Gamma$ corresponding to $u=1$ and $\gamma=1\in\ker\chi$ in \eqref{HHComputation tool}.
\end{definition}
It should be emphasised that being (un)twisted is a property of a pair $(\w,\Gamma)$, rather than its category of matrix factorisations. Indeed, we will see below that the polynomial $\w=x^3y+y^2$ is twisted and $\w=x^2y+y^2x$ is not, although Corollary \ref{HMS corollary} shows that the Hochschild cohomology of their respective categories of matrix factorisations are isomorphic. A pair $(\w,\Gamma)$ being untwisted ensures that all of the deformations corresponding to $\HH^2(Y_0)_{<0}$ come from semi-universal unfoldings of the polynomial $\w$. This is a key step in the proof of \cite[Theorem 1.6]{2018arXiv180604345L}, a special case of which appears as Theorem \ref{moduli of Ainfty}. By an abuse of notation, we will refer to a polynomial $\w$ as being (un)twisted to mean that the pair $(\w,\Gamma_{\w})$ is (un)twisted. 
\subsection{Loop polynomials} Consider $\W_0=x^py+y^qx$ with the only restriction that $p, q\geq 2$. Without loss of generality, we can consider $p\geq q$. This has weights as in \eqref{LoopWeights}, where we again set $d:=\gcd(p-1,q-1)$. As explained in Section \ref{background section}, we extend the action of $\Gamma_{\w}\simeq\C^*\times \mu_{d}$ to $\A^3$ as in \eqref{t0 weight} so that we now have
\begin{align}
\Gamma_{\w}=\{(t_0,t_1,t_2)\in(\C^*)^3|\ t_1^pt_2=t_2^qt_1=t_0t_1t_2\}.
\end{align} 
The group of characters is given by 
\begin{align}\label{LoopCharacters}
\hat{\Gamma}_\w:=\Hom(\Gamma_\w,\C^*)\simeq \Z\oplus \Z/d\Z,
\end{align}
and we take $m,n$ to be the same fixed solution to \eqref{LoopBezout} as in Section \ref{UnfoldingsofLoopPolys}.
We write each character $(t_0,t_1,t_2)\mapsto t_2^{mi-\frac{(q-1)j}{d}}t_1^{ni+\frac{(p-1)j}{d}}$, where $(i,j)\in\Z\oplus\Z/d\Z$, as $\rho_{i,j}$. One has that $\text{span}\{ z^\vee\}\simeq \rho_{\frac{(p-1)(q-1)}{d},0}$, $\text{span}\{x^\vee\}\simeq\rho_{\frac{(q-1)}{d},m}$, $ \text{span}\{y^\vee\}\simeq\rho_{\frac{(p-1)}{d},-n}$, $\chi\simeq \rho_{\frac{pq-1}{d},m-n}$, and $\ker\chi\simeq \mu_{pq-1}$. \\

We have that $\Jac_\w$ is given as in \eqref{Jac loop}. Since we are in the situation of an affine cone over an isolated hypersurface singularity, \cite[Section 3.1]{2018arXiv180604345L} shows that we must have $l=0$ in \eqref{HHComputation tool}. Furthermore, there are no contributions when $u<-1$, and the only possible contribution for $u=-1$ comes from when $N_\gamma=\text{span}\{x,y\}$, or $z\notin V_\gamma$. When $\gamma\in \ker\chi$ is the identity element, we have $V_\gamma=V$, $N_\gamma=0$, and $\W_\gamma=\w$. For every $u\in\Z_{\geq0}$, the elements
\begin{align*}
x^iy^jz^k&\in\Big(\Jac_\w\otimes\C[z]\otimes\chi^{\otimes u}\Big)^{\Gamma},\\
z^\vee\otimes x^iy^jz^{k+1}&\in\Big(z^\vee\otimes\Jac_\w\otimes\C[z]\otimes\chi^{\otimes u}\Big)^{\Gamma},
\end{align*}
where $i=u\bmod (p-1)$, $j=u\bmod (q-1)$, and $k=u+\lfloor\frac{u}{q-1}\rfloor+\lfloor\frac{u}{p-1}\rfloor$, contribute $\C(k)$ to $\HH^{2u}(Y_0)$ and $\HH^{2u+1}(Y_0)$, respectively. In addition, in the case where $u\equiv 0\bmod (p-1)$, the elements 
\begin{align*}
x^{p-1}y^{j}z^{k-1}&\in\Big(\Jac_\w\otimes\C[z]\otimes\chi^{\otimes u}\Big)^{\Gamma},\\
z^\vee\otimes x^{p-1}y^{j}z^{k}&\in\Big(z^\vee\otimes\Jac_\w\otimes\C[z]\otimes\chi^{\otimes u}\Big)^{\Gamma},
\end{align*}
where $i,j,$ and $k$ are as above, contribute $\C(k-1)$ to $\HH^{2u}(Y_0)$ and $\HH^{2u+1}(Y_0)$, respectively. In the case where $u\equiv 0\bmod (q-1)$, we also have the elements
\begin{align*}
x^{i}y^{q-1}z^{k-1}&\in\Big(\Jac_\w\otimes\C[z]\otimes\chi^{\otimes u}\Big)^{\Gamma},\\
z^\vee\otimes x^{i}y^{q-1}z^{k}&\in\Big(z^\vee\otimes\Jac_\w\otimes\C[z]\otimes\chi^{\otimes u}\Big)^{\Gamma},
\end{align*}
where $i,j,$ and $k$ are again as above, contribute $\C(k-1)$ to $\HH^{2u}(Y_0)$ and $\HH^{2u+1}(Y_0)$, respectively. In the case when $u\equiv 0\bmod \frac{(p-1)(q-1)}{d}$, we also have the elements 
\begin{align*}
x^{p-1}y^{q-1}z^{k-2}&\in\Big(\Jac_\w\otimes\C[z]\otimes\chi^{\otimes u}\Big)^{\Gamma},\\
z^\vee\otimes x^{p-1}y^{q-1}z^{k-1}&\in\Big(z^\vee\otimes\Jac_\w\otimes\C[z]\otimes\chi^{\otimes u}\Big)^{\Gamma},
\end{align*}
where $i,j,$ and $k$ are again as above, and these contribute $\C(k-2)$ to $\HH^{2u}(Y_0)$ and $\HH^{2u+1}(Y_0)$, respectively.\\

When $V_\gamma=0$, $N_\gamma=V$, $W_\gamma=0$, we have the summand
\begin{align*}
\big(\chi^{\vee}\otimes\wedge^{3}N_\gamma^\vee\big)^\Gamma\simeq \C\cdot x^\vee\wedge y^\vee\wedge z^\vee
\end{align*}
contributes $\C(-1)$ to $\HH^{2u+\dim N_\gamma}(Y_0)=\HH^{1}(Y_0)$, and there are $pq-d-1$ such $\gamma$. \\

In the case when $V_\gamma=\text{span}\{z\}$, $N_\gamma=\text{span}\{x,y\}$, $\W_\gamma=0$, we that for each $n\in\Z_{\geq0} $, the summands
\begin{align*}
\C \cdot z^{\frac{(n+1)(pq-1)}{d}-1}\otimes x^\vee\wedge y^\vee\simeq\Big(\Jac_{\W_\gamma}\otimes\chi^{\otimes {\frac{(n+1)(p-1)(q-1)}{d}-1}}\otimes \wedge^2N_\gamma^\vee\Big)^\Gamma,\\
\C\cdot  z^\vee\otimes z^{\frac{n(pq-1)}{d}}\otimes x^\vee\wedge y^\vee\simeq\Big(\Jac_{\W_\gamma}\otimes\chi^{\otimes \frac{n(p-1)(q-1)}{d}-1}\otimes \wedge^2N_\gamma^\vee\Big)^\Gamma,
\end{align*}
contribute $\C(\frac{(n+1)(pq-1)}{d}-1)$ to $\HH^{\frac{2(n+1)(p-1)(q-1)}{d}}(Y_0)$ and $\C(\frac{n(pq-1)}{d}-1)$ to $\HH^{\frac{2n(p-1)(q-1)}{d}+1}(Y_0)$. There are $d-1$ such contributions. \\

Putting this all together, we have that the Hochschild cohomology of $Y_0$ satisfies 
\begin{align}
\HH^{s+t}(Y_0)_t\simeq\HH^{s+t+2\frac{(p-1)(q-1)}{d}}(Y_0)_{s-\frac{pq-1}{d}}
\end{align}
for $s>0$, and that for $0\leq n\leq 2\frac{(p-1)(q-1)}{d}+1$, $\HH^{n}(Y_0)$ is given by 
\begin{align*}
&\HH^0(Y_0)\simeq\C(0),\\
&\HH^1(Y_0)\simeq\C(0)\oplus \C(-1)^{\oplus pq}\\
&\HH^{2r}(Y_0)\simeq \C\big(r+\lfloor\frac{r}{q-1}\rfloor+\lfloor\frac{r}{p-1}\rfloor\big)\quad \text{for } (p-1),(q-1)\nmid r\\
&\HH^{2r+1}(Y_0)\simeq\HH^{2r}(Y_0)\quad \text{for } (p-1),(q-1)\nmid r\\
&\HH^{2r(q-1)}(Y_0)\simeq \C\big(r(q-1)+\lfloor\frac{r(q-1)}{p-1}\rfloor+r\big)\oplus  \C\big(r(q-1)+\lfloor\frac{r(q-1)}{p-1}\rfloor+r-1\big)\quad\text{for }1\leq r<\frac{p-1}{d}\\
&\HH^{2r(q-1)+1}(Y_0)\simeq\HH^{2r(q-1)}(Y_0)\quad\text{for }1\leq r<\frac{p-1}{d}\\
&\HH^{2r(p-1)}(Y_0)\simeq \C\big(r(p-1)+\lfloor\frac{r(p-1)}{q-1}\rfloor+r\big)\oplus  \C\big(r(p-1)+\lfloor\frac{r(p-1)}{q-1}\rfloor+r-1\big)\quad\text{for }1\leq r<\frac{q-1}{d}\\
&\HH^{2r(p-1)+1}(Y_0)\simeq\HH^{2r(p-1)}(Y_0)\\
&\HH^{2\frac{(p-1)(q-1)}{d}}(Y_0)\simeq\HH^{2\frac{(p-1)(q-1)}{d}+1}(Y_0)\simeq \C(\frac{pq-1}{d})\oplus\C(\frac{pq-1}{d}-1)^{\oplus^{1+d}}\oplus\C(\frac{pq-1}{d}-2).
\end{align*}
Note that this is untwisted in every case. 
\subsection{Chain Polynomials}
Consider the case $\W_0=x^py+y^q$, where $p,q\geq 2$. This has weights as in \eqref{ChainWeights}, and we again take $d:=\gcd(p,q-1)$. We have $\Gamma_{\w}\simeq \C^*\times \mu_d$ as in \eqref{ChainMaxSymGp}, and extend the action to $\A^3$ as in \eqref{t0 weight} so that we now have
\begin{align}
\Gamma_{\w}=\{(t_0,t_1,t_2)\in (\C^*)^3|\ t_1^pt_2=t_2^q=t_0t_1t_2\}.
\end{align} 
The group of characters is given by 
\begin{align}\label{ChainCharacterGp}
\hat{\Gamma}_\w=\Hom(\Gamma_\w,\C^*)\simeq \Z\oplus \Z/d\Z,
\end{align}
and we take $m,n$ to be the same fixed solution to \eqref{ChainBezout} as in Section \ref{UnfoldingsofChainPolys}. We write each character $(t_0,t_1,t_2)\mapsto t_1^{ni+\frac{pj}{d}}t_2^{mi-\frac{(q-1)j}{d}}$ as $\rho_{i,j}$, where $(i,j)\in\Z\oplus\Z/d\Z$. One then has $ \text{span}\{z^\vee\}\simeq\rho_{\frac{(p-1)(q-1)}{d},0}$, $\text{span}\{x^\vee\}\simeq\rho_{\frac{q-1}{d},m}$, $\text{span}\{ y^\vee\}\simeq \rho_{\frac{p}{d},-n}$, $\chi\simeq \rho_{\frac{pq}{d},m-n}$, $\ker\chi\simeq\mu_{pq}$.\\

We have that $\Jac_\w$ is given as in \eqref{Jac chain}. As in the loop case, we have $l=0$ and $u\geq-1$ in \eqref{HHComputation tool}, where $u=-1$ only if $N_\gamma=\text{span}\{x,y\}$, or $z\notin V_\gamma$. In the case where $\gamma\in\ker\chi$ is the identity, we have $V_\gamma=V$, $N_\gamma=0$, and $\W_\gamma=\w$. For each $u\in\Z_{\geq0}$, we have that the elements 
\begin{align*}
x^{i}y^{j}z^k&\in\Big(\Jac_\w\otimes\C[z]\otimes\chi^{\otimes u}\Big)^{\Gamma},\\
z^\vee\otimes x^{i}y^{j}z^{k+1}&\in\Big(z^\vee\otimes\Jac_\w\otimes\C[z]\otimes\chi^{\otimes u}\Big)^{\Gamma},
\end{align*}
where $j=u\bmod (q-1)$, $i=\frac{upq-jp}{q-1}\bmod (p-1)$, and $k=\frac{upq-i(q-1)-jp}{(p-1)(q-1)}$, contribute $\C(k)$ to $\HH^{2u}(Y_0)$ and $\HH^{2u+1}(Y_0)$, respectively. In addition, when $u\equiv 0\bmod (q-1)$, we have contributions from the elements
\begin{align*}
x^{i'}y^{q-1}z^{k'}&\in\Big(\Jac_\w\otimes\C[z]\otimes\chi^{\otimes u}\Big)^{\Gamma},\\
z^\vee\otimes x^{i'}y^{q-1}z^{k'+1}&\in\Big(z^\vee\otimes\Jac_\w\otimes\C[z]\otimes\chi^{\otimes u}\Big)^{\Gamma},
\end{align*}
where $i'=\frac{upq-(q-1)p}{q-1}\bmod (p-1)$ and $k'=\frac{upq-i'(q-1)-(q-1)p}{(p-1)(q-1)}$, and these contribute $\C(k')$ to $\HH^{2u}(Y_0)$ and $\HH^{2u+1}(Y_0)$, respectively.\\

In the case where $u\equiv 0\bmod\frac{(p-1)(q-1)}{\gcd(p-1,q)}$, we also have  
\begin{align*}
x^{p-1}z^k&\in\Big(\Jac_\w\otimes\C[z]\otimes\chi^{\otimes u}\Big)^{\Gamma},\\
z^\vee\otimes x^{p-1}z^{k+1}&\in\Big(z^\vee\otimes\Jac_\w\otimes\C[z]\otimes\chi^{\otimes u}\Big)^{\Gamma},
\end{align*}
where $k=\frac{upq}{(p-1)(q-1)}-1$. These contribute to $\C(k)$ to $\HH^{2u}(Y_0)$ and $\HH^{2u+1}(Y_0)$, respectively. \\

For the elements $\gamma\in\ker\chi$ such that $V_\gamma=0$, $N_\gamma=V$, and $\W_\gamma=0$, we have that the only contribution is from the summand
\begin{align*}
\big(\chi^{\vee}\otimes\wedge^{3}N_\gamma^\vee\big)^\Gamma\simeq \C\cdot x^\vee\wedge y^\vee\wedge z^\vee, 
\end{align*}
which contributes $\C(-1)$ to $\HH^{2u+\dim N_\gamma}(Y_0)=\HH^{1}(Y_0)$, and there are $pq-p-\gcd(p-1,q)+1$ such $\gamma$. \\

In the case where $V_\gamma=\text{span}\{y\}$, there cannot be a contribution. There are $p-1$ such elements of $\ker\chi$ which fix $y$ and nothing else. \\

In the case where $V_\gamma=\text{span}\{z\}$, $N_\gamma=\text{span}\{x,y\}$, we have for each $n\in\Z_{\geq0}$, there are contributions from the summands 
\begin{align*}
\C\cdot z^{\frac{(n+1)pq}{\gcd(p-1,q)}-1}\otimes x^\vee\wedge y^\vee&\simeq\Big(\Jac_{\W_\gamma}\otimes \chi^{\frac{(n+1)(p-1)(q-1)}{\gcd(p-1,q)}-1}\otimes\wedge^2N_\gamma^\vee\Big)^\Gamma,\\
\C\cdot z^\vee\otimes z^{\frac{npq}{\gcd(p-1,q)}}\otimes x^\vee\wedge y^\vee&\simeq\Big(z^\vee\otimes\Jac_{\W_\gamma}\otimes \chi^{\frac{n(p-1)(q-1)}{\gcd(p-1,q)}-1}\otimes\wedge^2N_\gamma^\vee\Big)^\Gamma,
\end{align*}
and these contribute $\C(\frac{(n+1)pq}{\gcd(p-1,q)}-1)$ to $\HH^{\frac{2(n+1)(p-1)(q-1)}{\gcd(p-1,q)}}(Y_0)$ and $\C(\frac{npq}{\gcd(p-1,q)}-1)$ to  $\HH^{2\frac{n(p-1)(q-1)}{\gcd(p-1,q)}+1}(Y_0)$. There are $\gcd(p-1,q)-1$ such terms.
In total, we have that 
\begin{align}
\HH^{s+t}(Y_0)_t\simeq \HH^{s+t+2\frac{(p-1)(q-1)}{\gcd(p-1,q)}}(Y_0)_{t-\frac{pq}{\gcd(p-1,q)}}
\end{align} 
for $s>0$, and for $0\leq n\leq 2\frac{(p-1)(q-1)}{\gcd(p-1,q)}+1$, $\HH^n(Y_0)$ is given by
\begin{align*}
&\HH^0(Y_0)\simeq\C(0)\\
&\HH^1(Y_0)\simeq\C(0)\oplus\C(-1)^{\oplus(p(q-1)+1)}\\
&\HH^{2r}(Y_0)\simeq \C(\lfloor\frac{rp}{p-1}\rfloor)\quad\text{for } (q-1)\nmid r\\
&\HH^{2r+1}(Y_0)\simeq\HH^{2r}(Y_0)\quad\text{for } (q-1)\nmid r\\
&\HH^{2r(q-1)}(Y_0)\simeq \C(\lfloor\frac{rp(q-1)}{p-1}\rfloor)\oplus\C(\lfloor\frac{p(rq-1)}{p-1}\rfloor)\quad\text{for } 1\leq r<\frac{p-1}{\gcd(p-1,q)}\\
&\HH^{2r(q-1)+1}(Y_0)\simeq\HH^{2r(q-1)}(Y_0)\text{ for } 1\leq r<\frac{p-1}{\gcd(p-1,q)}\\
&\HH^{2\frac{(p-1)(q-1)}{\gcd(p-1,q)}}(Y_0)\simeq \C(\frac{pq}{\gcd(p-1,q)})\oplus\C(\frac{pq}{\gcd(p-1,q)}-1)^{\oplus\gcd(p-1,q)}\oplus \C(\frac{pq}{\gcd(p-1,q)}-2)\\
&\HH^{2\frac{(p-1)(q-1)}{\gcd(p-1,q)}+1}(Y_0)\simeq\HH^{2\frac{(p-1)(q-1)}{\gcd(p-1,q)}}(Y_0).
\end{align*}
This is twisted for the $(p,q)=(3,2)$, but is otherwise untwisted. 
\subsection{Brieskorn--Pham Polynomials}
Consider $\W_0=x^p+y^q$, and without loss of generality, that $p\geq q\geq 2$. We are excluding the case of $p=q=2$, since $d_0=0$ in this case. This has weights as in \eqref{BPWeights}, where we again set $d:=\gcd(p,q)$. We have $\Gamma_{\w}\simeq \C^*\times\mu_d$, as in \eqref{BPMaxSymGp}, and extend the action to $\A^3$ as in \eqref{t0 weight}, so that we now have
\begin{align}
\Gamma_{\w}=\{(t_0,t_1,t_2)\in(\C^*)^3|\ t_1^p=t_2^q=t_0t_1t_2\}.
\end{align} 
The group of characters is given by 
\begin{align}
\hat{\Gamma}_\w:=\Hom(\Gamma_\w,\C^*)\simeq \Z\oplus \Z/d\Z,
\end{align}
and we again take $m,n$ to be the same fixed solution to \eqref{BPBezout} as in Section \ref{UnfoldingsofBPPolys}. We write each character $(t_0,t_1,t_2)\mapsto t_2^{mi-\frac{qj}{d}}t_1^{ni+\frac{pj}{d}}$, where $(i,j)\in\Z\oplus\Z/d\Z$, as $\rho_{i,j}$. One has that $\text{span}\{z^\vee\}\simeq \rho_{\frac{(p-1)(q-1)-1}{d},d-m+n}$, $\text{span}\{ x^\vee\}\simeq\rho_{\frac{q}{d},m}$, $\text{span}\{ y^\vee\}\simeq\rho_{\frac{p}{d},-n}$, $\chi\simeq \rho_{\frac{pq}{d},0}$, and $\ker\chi\simeq \mu_\frac{pq}{d}\times\mu_d$. \\

We have that $\Jac_\w$ is given as in \eqref{Jac BP}. As in the loop and chain cases, we have $l=0$ and $u\geq-1$ in \eqref{HHComputation tool}, where $u=-1$ only if $N_\gamma=\text{span}\{x,y\}$, or $z\notin V_\gamma$. When $\gamma\in\ker\chi$ is the identity, we have that for $0\leq u\leq \frac{(p-1)(q-1)-1}{d}$, the elements
\begin{align*}
x^{i}y^{j}z^k&\in\Big(\Jac_\w\otimes\C[z]\otimes\chi^{\otimes u}\Big)^{\Gamma},\\
z^\vee\otimes x^{i}y^{j}z^{k+1}&\in\Big(z^\vee\otimes\Jac_\w\otimes\C[z]\otimes\chi^{\otimes u}\Big)^{\Gamma},
\end{align*}
where $i,j,k$ are solutions to
\begin{align}\label{equation for k}
\begin{split}
i-k&=-mp\\
j-k&=-nq\\
k&=u+m+n\\
0\leq i&\leq p-2\\
0\leq j&\leq q-2,
\end{split}
\end{align}
contribute $\C(k)$ to $\HH^{2u}(Y_0)$, and $\HH^{2u+1}(Y_0)$. In the case where $u=\frac{(p-1)(q-1)-1}{d}$, we have that there are precisely two solutions to \eqref{equation for k}, otherwise the solution is unique. \\

For the elements $\gamma\in\ker\chi$ such that $V_\gamma=0$, $N_\gamma=V$, and $\W_\gamma=0$, we have that the only contribution is from the summand
\begin{align*}
\big(\chi^{\vee}\otimes\wedge^{3}N_\gamma^\vee\big)^\Gamma\simeq  \C\cdot x^\vee\wedge y^\vee\wedge z^\vee,
\end{align*}
and this contributes $\C(-1)$ to $\HH^{2u+\dim N_\gamma}(Y_0)=\HH^{1}(Y_0)$. There are $(p-1)(q-1)-\gcd(p,q)+1$  such $\gamma$. \\

When $V_\gamma=\text{span}\{x\}$ or $V_\gamma=\text{span}\{y\}$, there is no contribution. There are $q-1$ and $p-1$ such elements in $\ker\chi$, respectively. \\

When $V_\gamma=\text{span}\{z\}$, $N_\gamma=\text{span}\{x,y\}$, $\W_\gamma=0$ for $n\geq 0$ we have that the summands
\begin{align*}
\C\cdot z^{\frac{(n+1)pq}{d}-1}\otimes x^\vee\wedge y^\vee&\simeq\Big(\Jac_{\W_\gamma}\otimes \chi^{\frac{(n+1)(p-1)(q-1)}{d}-1}\otimes\wedge^2N_\gamma^\vee\Big)^\Gamma,\\
\C\cdot z^\vee\otimes z^{\frac{npq}{d}}\otimes x^\vee\wedge y^\vee&\simeq\Big(z^\vee\otimes\Jac_{\W_\gamma}\otimes \chi^{\frac{n(p-1)(q-1)}{d}-1}\otimes\wedge^2N_\gamma^\vee\Big)^\Gamma,
\end{align*}
contribute $\C(\frac{(n+1)pq}{d}-1)$ and $\C(\frac{npq}{d}-1)$ to $\HH^{\frac{2(n+1)((p-1)(q-1)-1)}{d}}(Y_0)$ and $\HH^{\frac{2n((p-1)(q-1)-1)}{d}+1}(Y_0)$, respectively.
There are $\gcd(p,q)-1$ such terms. Putting this all together, we get that
\begin{align}
\HH^{s+t}(Y_0)_t\simeq\HH^{s+t+2\frac{(p-1)(q-1)-1}{d}}(Y_0)_{t-\frac{pq}{d}}
\end{align} 
for $s>0$, and that for $0\leq n\leq \frac{2(p-1)(q-1)-1}{d}+1$, we have that $\HH^n(Y_0)$ is given by
\begin{align*}
&\HH^0(Y_0)\simeq\C(0)\\
&\HH^1(Y_0)\simeq\C(0)\oplus\C(-1)^{\oplus(p-1)(q-1)}\\
&\HH^{2r}(Y_0)\simeq\HH^{2r+1}(Y_0)\simeq \C(k)\quad\text{for $r<\frac{(p-1)(q-1)-1}{\gcd(p,q)}$ and $k$ the unique solution to \eqref{equation for k}}\\
&\HH^{2\frac{(p-1)(q-1)-1}{\gcd(p,q)}}(Y_0)\simeq \C(\frac{pq}{\gcd(p,q)}-2)\oplus \C(\frac{pq}{\gcd(p,q)}-1)^{\oplus \gcd(p,q)-1}\oplus \C(\frac{pq}{\gcd(p,q)})\\
&\HH^{2\frac{(p-1)(q-1)-1}{\gcd(p,q)}}(Y_0)\simeq\HH^{2\frac{(p-1)(q-1)-1}{\gcd(p,q)}+1}(Y_0).
\end{align*}
Note that this is twisted in the case $p=q=3$, and $p=4,\ q=2$, but is otherwise untwisted. 
\subsection{Unfoldings of invertible polynomials}\label{HH for unfoldings of invertible polynomials}
Of course, Theorem \ref{HH comp tool thm} can also be used to compute the Hochschild cohomology of the category of matrix factorisations of an unfolded polynomial. For the polynomials where $\dim U_+>1$, we will need Hochschild cohomology calculations of unfolded polynomials in order to be able to isolate the correct mirror. Towards this end we will need to (at least partially) calculate $\HH^2(Y_u)$ in these cases.

\begin{lem}\label{RulingOutDeformationsLemma}
Let $\w$ an untwisted invertible polynomial in two variables such that $\dim U_+>1$. Then:
\begin{itemize}
	\item For $\w=x^py+y^2x$ and $p>2$, we have $\mathrm{HH}^2(Y_u)=0$ unless $u_{1,0}=0$, $u_{1,1}\neq0$ in \eqref{LoopUnfolding}.
	\item For $\w=x^2y+xy^2$, we have $\dim\mathrm{HH}^2(Y_u)<3$ unless $u_{1,1}\neq 0$ and $u_{0,0}=u_{1,0}=u_{0,1}=0$ in \eqref{LoopUnfolding}. 
	\item For $\w=x^py+y^2$ and $p>3$, we have $\mathrm{HH}^2(Y_u)=0$ unless $u_{1,1}\neq 0 $ and $u_{2,0}=0$ in \eqref{ChainUnfolding}
	\item For $\w=x^2y+y^2$, we have $\dim \mathrm{HH}^2(Y_u)<2$ unless $u_{0,1}\neq 0$ and $u_{0,0}=u_{1,0}=0$ in \eqref{ChainUnfolding}.
\end{itemize}
\end{lem}
\begin{proof}
In each of the cases we consider, the sequence $(\partial_x\W_u,\partial_y\W_u)$ is a regular sequence in $S$. Therefore, the cohomology of the Koszul complex, \eqref{Koszul cpx}, will be concentrated in degrees $0$ and $-1$, and the only contributions to $\HH^2(Y_u)$ can come from $(\Jac_{\W_u}\otimes \chi)^\Gamma$ and $(\Jac_{\W_u}\otimes x^\vee\wedge y^\vee)^\Gamma$. Note that if the latter term contributes to $\HH^2(Y_u)$, then the polynomial is twisted, and we will not consider it.\\

The two loop polynomials we must consider are $\w=x^py+y^2x$ for $p>2$ and $\w=x^2y+y^2x$. In the former case, the unfolding is given by $\W_u= x^py+y^2x+u_{1,1}xyz+u_{1,0}xz^2$. For a contribution to $\HH^2(Y_u)$, there must be an element of $\Jac_{\W_u}$ which is proportional to $\chi$. Note that if $u_{1,1}=0$ then $\dim (\Jac_{\W_u}\otimes \chi)^\Gamma=0$. On the other hand, we have that $\dim (\Jac_{\W_u}\otimes \chi)^\Gamma=0$ if $u_{1,0}\neq0$.
In the case $\w=x^2y+y^2x$, we have that $\dim (\Jac_{\W_u}\otimes \chi)^\Gamma<3$ unless $u_{1,1}\neq 0$, and the other coefficients are zero. \\

The only chain polynomials which need to be considered are $\w=x^py+y^2$ for $p>3$ and $\w=x^2y+y^2$. In the former case, note that if $u_{1,1}=0$, or $u_{1,1},u_{2,0}\neq0$, then $\HH^2(Y_u)=0$.  In the latter case, note that $\dim\HH^2(Y_u)<2$ unless $u_{0,1}\neq 0$ and the other coefficients are zero. 
\end{proof}
\section{Generators and formality}\label{generators and formality}
In this section we recall and implement the results of various authors to establish the required generation statements for the compact Fukaya category of the Milnor fibre, and also the category of perfect complexes on $Y_u$ for any $u\in U_+$, as outlined in Section \ref{IntroStrategy}.\\

As in the previous sections, let $\check{V}_{\check{\w}}$ be the Milnor fibre of the transpose of an invertible polynomial in two variables such that $d_0>0$. Let $\{S_i\}_{i=1}^{\check{\mu}}$ be a distinguished basis of vanishing cycles, and let $\mathcal{S}$ be the full subcategory of $\mathcal{F}(\check{V}_{\check{\w}})$ whose objects are $\{S_i\}_{i=1}^{\check{\mu}}$. As in Section \ref{IntroStrategy}, denote by $\mathcal{A}$ the total $A_\infty$-endomorphism algebra of $\mathcal{S}$,
\begin{align}
\mathcal{A}:=\bigoplus_{i,j}^{\check{\mu}}\text{hom}_{\mathcal{F}(\check{V}_{\check{\w}})}(S_i,S_j).
\end{align}
Let $T_{L}\in \text{Symp}(\Sigma;\partial\Sigma)$ be the Dehn twist around a Lagrangian $L$ in a surface with boundary $(\Sigma;\partial\Sigma)$, as in \cite[Section 16c]{SeidelBook}. By \cite[Theorem 4.17, Comment 4.18(c)]{BSMF_2000__128_1_103_0}, we have that 
\begin{align}
\big(T_{S_1}\circ\dots\circ T_{S_{\check{\mu}}}\big)^{\check{h}}=[2\check{d}_0].
\end{align}
Since $\check{d}_0>0$, the argument of \cite[Lemma 5.4]{HMSQuartic} then shows that $\mathcal{S}$ split-generates $\mathcal{F}(\check{V}_{\check{\w}})$, and so
\begin{align}
\mathcal{F}(\check{V}_{\check{\w}})\simeq\perf \mathcal{S}.
\end{align}
On the B--side, let $\w:\A^2\rightarrow\A$ be an invertible polynomial in two variables such that $d_0>0$. In each case, we aim to associate $U_+$ to the moduli space of $A_\infty$-structures on a fixed quiver algebra. In order to do this, for each $u\in U_+$ we must find generators $\mathcal{S}_u$ of $\perf Y_u$ such that
\begin{enumerate}[(i)]
	\item the isomorphism class of the cohomology level endomorphism algebra $\End(\S_u)$ does not depend on $u\in U_+$, and
	\item the generator $\S_0$ at $0\in U_+$ admits a $\C^*$-equivariant structure such that the cohomological grading on $\End(\S_u)$ is proportional to the weight of the $\C^*$-action. 
\end{enumerate}
If we find generators which satisfy condition (i), then we can think of deformations of $Y$ in terms of deformations of the $A_\infty$-structures on the cohomology level endomorphism algebra. Condition (ii) will be necessary to deduce that $\text{end}(\mathcal{S}_0)$ is formal. \\

Recall (\cite[Theorem 2]{HabermannSmith}) that $\mf(\A^2,\Gamma_{\w},\w)$ has a tilting object, $\mathcal{E}$, for any two variable invertible polynomial $\w$. For each $u\in U_+$, let $\mathcal{S}_u$ be the image of $\mathcal{E}$ under the pushforward functor 
\begin{align*}
\mf(\A^2,\Gamma,\w)\rightarrow \mf(\A^3,\Gamma,\W_u)\simeq \coh Y_u.
\end{align*}
It is then a consequence of \cite[Theorem 4.1]{2018arXiv180604345L} that $\mathcal{S}_u$ split-generates $\perf Y_u$.\\

Let $\mathcal{A}_u$ be the minimal model of the dg-endomorphism algebra of $\mathcal{S}_u$, $\operatorname{end}(\S_u)$. As discussed in Section \ref{HHcomputation section}, one has a quasi-equivalence
\begin{align}
\Qcoh Y_u\simeq \text{Mod}\,\mathcal{A}_u,
\end{align}
and therefore, by the Morita invariance of Hochschild cohomology, an isomorphism 
\begin{align}
\HH^*(Y_u)\simeq \HH^*(\mathcal{A}_u).
\end{align}
The cohomology algebra $A_u:=H^*(\mathcal{A}_u)$ is independent of $u$, and by \cite[Theorem 1.1]{Hyperplanesections}, is isomorphic as a vector space to \eqref{TEA VS}. On both the A--, and B--sides, the algebra structure is given as in \eqref{TEA multiplication}, since $A^\rightarrow$ is the quiver algebra of a quiver with no cycles, and so $\HH^2(A^\rightarrow,(A^\rightarrow)^\vee[-1])=(\HH_1(A^\rightarrow))^\vee=0$.\\

By exploiting the additional $\C^*$-action, one can prove a general statement for the formality of $\mathcal{A}_0$. This is done by first showing that the cohomological grading on $\End^*(\mathcal{S}_0)$ is proportional (equal in the case of curves) to the weight of the $\C^*$-action. This follows from the fact that the dualising sheaf of $Y_0$ is trivial as an $\mathcal{O}_{Y_0}$-module, but has weight one with respect to the additional $\C^*$-action. Since $\C^*$ is reductive, the chain homotopy to take $\text{end}(\mathcal{S}_0)$ to a minimal $A_\infty$-structure can be made $\C^*$-equivariant. Since $\mu^d$ lowers the cohomological degree by $2$, the only map which can be non-zero is $\mu^2$. 
\begin{thm}[{\cite[Theorem 4.2]{2018arXiv180604345L}}]
	$\mathcal{A}_0$ is formal.
\end{thm} 
In particular, this means that
\begin{align}\label{Formal HH}
\HH^*(Y_0)\simeq \HH^*(A),
\end{align}
and so the computations in Section \ref{HHcomputation section} imply that the moduli space of $A_\infty$-structures on $A$ is represented by an affine scheme of finite type. Furthermore, combining equation \eqref{Formal HH} with Theorem \ref{HHSHisom}, and the calculations in Section \ref{SHMilnorFibre} gives us that the $A_\infty$-structure on $\mathcal{A}$, the endomorphism algebra of the generators of $\mathcal{F}(\check{V}_{\check{\w}})$, is \emph{not} formal.
\section{Homological mirror symmetry for invertible polynomials in two variables}
In this section, we bring together the previous sections of the paper to establish Theorem \ref{main theorem} and Corollary \ref{HMS corollary}. As noted above, the computations of Section \ref{HHcomputation section} together with \eqref{Formal HH} mean that the moduli space of $A_\infty$-structures on $A$ is represented by an affine scheme of finite type, $\mathcal{U}_\infty(A)$, for any untwisted invertible polynomial $\w$. As explained in Section \ref{IntroStrategy}, we would like to identify $\mathcal{U}_{\infty}(A)$ with the space $U_+$ corresponding to $\w$ by showing that the map \eqref{U+ModuliMap} is an isomorphism. To this end, we utilise the following special case of \cite[Theorem 1.6]{2018arXiv180604345L}:
\begin{thm}\label{moduli of Ainfty}
	Let $\w$ be an untwisted invertible polynomial in two variables such that $d_0>0$, and $\Gamma$ be a subgroup of $\Gamma_\w$ containing $\phi(\C^*)$ as a subgroup of finite index. Let $A^\rightarrow$ be the endomorphism algebra of a tilting object in $\mathrm{mf}(\A^2,\Gamma,\w)$, and let $A$ be the degree $1$ trivial extension algebra of $A^\rightarrow$. Then there is a $\C^*$-equivariant isomorphism $U_+\xrightarrow{\sim}\mathcal{U}_\infty(A)$ which sends $0\in U_+$ to the formal $A_\infty$-structure on $A$. 
\end{thm}
This isomorphism descends to the quotient by the $\C^*$-action, and so we get an isomorphism $\big(U_+\setminus(\boldsymbol{0})\big)/\C^*\xrightarrow{\sim}\mathcal{M}_\infty(A) $. It should be reiterated that the polynomial being untwisted is a crucial assumption, as can be seen by considering, for example, $\w=x^3y+y^2$. In this case, we have that $\HH^2(Y_0)_{<0}= \C(3)\oplus \C(2)^{\oplus 2}\oplus \C(1)$, but $U_+=\A^3$. 

\begin{proof}[Proof of Theorem \ref{main theorem}]

In each case, we know that the $A_\infty$-structure on $\mathcal{F}(\check{V}_{\check{\w}})$ is not formal, and so is represented by a point in $\mathcal{M}_\infty(A)$. By Theorem \ref{moduli of Ainfty}, this, in turn, represents the $A_\infty$-structure corresponding to the dg-enhancement of the derived category of perfect sheaves on a semi-universal unfolding of $\w$. In the cases where $\dim U_+=1$, we have that $\mathcal{M}_\infty(A)$ is a single point, and so the semi-universal unfolding (up to scaling) corresponding to this point must be the mirror. Note that in the cases $\w=x^2y+y^q$ for $q>2$ and $\w=x^p+y^2$ for $p>4$, we have
\begin{align*}
\C[x,y,z]/(x^2y+y^q+yz^2)&\simeq\C[x,y,z]/(x^2y+y^q+xyz),\\
\C[x,y,z]/(x^p+y^2+x^2z^2)&\simeq\C[x,y,z]/(x^p+y^2+xzy)
\end{align*}
by completing the square. \\

In the case where $\dim\,U_+>1$, we must exclude the points in $\mathcal{M}_\infty(A)$ other than the claimed mirror. In the case $\w=x^py+y^2$ for $p>3$, we have by Lemma \ref{RulingOutDeformationsLemma} that $\dim \HH^2(Y_u)=0<\dim \SH^2(\check{V}_{\check{\w}})$ unless $u=(0,1)$. By Theorem \ref{HHSHisom}, we must therefore have that the mirror is identified with $Y_u$ for $u=(0,1)\in U_+$. A similar argument in the cases $\w=x^2y+y^2x$ and $\w=x^py+xy^2$ for $p>2$ leads to identifying the mirrors as $Y_u$ for $u=(0,0,0,1)$ and $u=(0,1)$, respectively. \\

In the case of $x^2y+y^2$, we have that if $u\neq (0,0,1)$, then $\dim \HH^2(Y_u)<2=\dim \SH^2(\check{V}_{\check{\w}})$ by Lemma \ref{RulingOutDeformationsLemma}, and so the mirror is identified with $Y_u$ for $u=(0,0,1)$. Again, by completing the square, we have
\begin{align*}
\C[x,y,z]/(x^2y+y^2+yz^2)\simeq\C[x,y,z]/(x^2y+y^2+xyz).
\end{align*}

In the case of $\w=x^3+y^2$, we follow the same argument as in \cite{DehnSurgery}. Namely, we have that if $Y_u$ is an elliptic curve, then $\HH^*(Y_u)$ exists in only finitely many degrees by the Hochschild--Kostant--Rosenberg theorem. Since the symplectic cohomology of the Milnor fibre is non-trivial in arbitrarily large degree, by Theorem \ref{HHSHisom}, we have that the mirror cannot be smooth. We therefore have that the mirror must be the nodal cubic $\W_u=x^3+y^2+xz^4+\frac{\sqrt[3]{2}z^6}{\sqrt{3}}$, and we have
\begin{align*}
\C[x,y,z]/(\W_u)\simeq \C[x,y,z]/(x^3+y^2+xyz)
\end{align*}
by a change of variables.\\

In the cases where the polynomial is twisted, this result has already been established in \cite{Auslanderorders} by different means. Our construction of the Milnor fibres agrees with the surfaces constructed in \cite{Auslanderorders}, and the mirrors established there are precisely the mirrors we claim.\\

The only invertible polynomial where $d_0\not>0$ is $\w=x^2+y^2$, for which $d_0=0$. This, however, corresponds to the mirror symmetry statement for $\C^*$, which is already well established. Therefore, Theorem \ref{main theorem} is true in this case, too. 

\end{proof}
\begin{proof}[Proof of Corollary \ref{HMS corollary}]
By observing that the results of Section \ref{graded homeos between milnor fibres} show that the relevant compact Fukaya categories are quasi-equivalent,  Theorem \ref{main theorem} establishes that the derived categories of perfect complexes of their mirrors are, too. 
\end{proof}

\bibliography{2_var_HMSbib}

\providecommand{\href}[2]{#2}\begingroup\raggedright\begin{thebibliography}{10}

\bibitem{Atiyah}
M.~F. Atiyah, ``{R}iemann surfaces and spin structures,''
  \href{http://dx.doi.org/10.24033/asens.1205}{{\em Ann. Sci. {\'E}c. Norm.
  Sup{\'e}r.} {\bfseries Ser. 4, 4} no.~1, (1971) 47--62}.

\bibitem{BFK}
M.~Ballard, D.~Favero, and L.~Katzarkov, ``A category of kernels for
  equivariant factorizations and its implications for {H}odge theory,''
  \href{http://dx.doi.org/10.1007/s10240-013-0059-9}{{\em Publ. Math. Inst.
  Hautes {\'E}tudes Sci.} {\bfseries 120} (05, 2011) 1–--111}.

\bibitem{BFKVGIT}
M.~Ballard, D.~Favero, and L.~Katzarkov, ``Variation of geometric invariant
  theory quotients and derived categories,''
  \href{http://dx.doi.org/10.1515/crelle-2015-0096}{{\em J. Reine Angew. Math.}
  {\bfseries 2019} no.~746, (03, 2019) 235--303}.

\bibitem{BerglundHubsch}
P.~Berglund and T.~H\"{u}bsch, ``A generalized construction of mirror
  manifolds,''
  \href{http://dx.doi.org/https://doi.org/10.1016/0550-3213(93)90250-S}{{\em
  Nuclear Physics B} {\bfseries 393} no.~1, (1993) 377--391}.

\bibitem{ChoChoaJeong}
C.-H. {Cho}, D.~{Choa}, and W.~{Jeong}, ``{Fukaya category for
  {L}andau-{G}inzburg orbifolds and {B}erglund-{H}{\"u}bsch conjecture for
  invertible curve singularities},'' {\em Preprint, arXiv:2010.09198} (2020) .

\bibitem{CaldararuTu}
A.~C\u{a}ld\u{a}raru and J.~Tu, ``Curved ${A}_\infty$-algebras and
  {L}andau-{G}inzburg models,'' {\em New York J. Math.} {\bfseries 19} (07,
  2010) .

\bibitem{2016arXiv160106027E}
W.~{Ebeling}, ``{Homological mirror symmetry for singularities},'' {\em
  Preprint, arXiv: 1601.06027} .

\bibitem{10.1093/imrn/rns115}
W.~Ebeling and A.~Takahashi, ``{Mirror symmetry between orbifold curves and
  cusp singularities with group action},''
  \href{http://dx.doi.org/10.1093/imrn/rns115}{{\em Int. Math. Res. Not. IMRN}
  {\bfseries 2013} no.~10, (04, 2012) 2240--2270}.

\bibitem{FAVERO2019943}
D.~Favero and T.~L. Kelly, ``Derived categories of {BHK} mirrors,''
  \href{http://dx.doi.org/https://doi.org/10.1016/j.aim.2019.06.013}{{\em Adv.
  Math.} {\bfseries 352} (2019) 943 -- 980}.

\bibitem{Futaki2011}
M.~Futaki and K.~Ueda, ``Homological mirror symmetry for {B}rieskorn--{P}ham
  singularities,'' \href{http://dx.doi.org/10.1007/s00029-010-0055-6}{{\em Sel.
  Math. New Ser.} {\bfseries 17} no.~2, (2011) 435--452}.

\bibitem{FU2}
M.~Futaki and K.~Ueda, ``Homological mirror symmetry for singularities of type
  {D},'' \href{http://dx.doi.org/10.1007/s00209-012-1024-x}{{\em Math. Z.}
  {\bfseries 273} (04, 2013) 633--652}.

\bibitem{Gammage}
B.~{Gammage}, ``{{M}irror symmetry for {B}erglund-{H}{\"u}bsch {M}ilnor
  fibers},'' {\em Preprint, arXiv:2010.15570} (2020) .

\bibitem{HabermannSmith}
M.~{Habermann} and J.~{Smith}, ``{Homological Berglund-H{\"u}bsch mirror
  symmetry for curve singularities},'' {\em J. Symplectic Geom.} {\bfseries 18}
  no.~6, (2020) 1515--1574.

\bibitem{HLVGIT}
D.~Halpern-Leistner, ``The derived category of a {GIT} quotient,''
  \href{http://dx.doi.org/10.1090/S0894-0347-2014-00815-8}{{\em J. Amer. Math.
  Soc.} {\bfseries 28} (03, 2015) 871--912}.

\bibitem{HiranoOuchi}
Y.~{Hirano} and G.~{Ouchi}, ``{Derived factorization categories of
  non-{T}hom--{S}ebastiani-type sum of potentials},'' {\em Preprint, arXiv:
  1809.09940} .

\bibitem{hirzebruch20060}
F.~Hirzebruch and K.~Mayer, {\em {O}(n) - Mannigfaltigkeiten, exotische
  Sph{\"a}ren und Singularit{\"a}ten}.
\newblock Lecture Notes in Mathematics. Springer Berlin Heidelberg, 2006.

\bibitem{HopfHeinz1983DGit}
H.~Hopf, {\em Differential geometry in the large}.
\newblock Lecture notes in mathematics 1000, Springer Verlag, Berlin. 1983.

\bibitem{Isikorlov}
M.~Isik, ``Equivalence of the derived category of a variety with a singularity
  category,'' \href{http://dx.doi.org/10.1093/imrn/rns125}{{\em Int. Math. Res.
  Not. IMRN} {\bfseries 2013} (2013) 2787–--2808}.

\bibitem{JohnsonSpinstructures}
D.~Johnson, ``Spin structures and quadratic forms on surfaces,''
  \href{http://dx.doi.org/10.1112/jlms/s2-22.2.365}{{\em J. London Math. Soc.}
  {\bfseries s2-22} no.~2, (1980) 365--373}.

\bibitem{2019arXiv191109859K}
O.~{Kravets}, ``{Categories of singularities of invertible polynomials},'' {\em
  Preprint, arXiv: 1911.09859} .

\bibitem{Krawitz}
M.~Krawitz, {\em F{JRW} rings and {L}andau-{G}inzburg mirror symmetry}.
\newblock ProQuest LLC, Ann Arbor, MI, 2010.
\newblock Thesis (Ph.D.)--University of Michigan.

\bibitem{kreuzer1992}
M.~Kreuzer and H.~Skarke, ``On the classification of quasihomogeneous
  functions,'' {\em Comm. Math. Phys.} {\bfseries 150} no.~1, (1992) 137--147.

\bibitem{2012arXiv1211.4632L}
Y.~{Lekili} and T.~{Perutz}, ``{Arithmetic mirror symmetry for the
  $2$-torus},'' {\em Preprint, arXiv: 1211.4632} .

\bibitem{DehnSurgery}
Y.~Lekili and T.~Perutz, ``Fukaya categories of the torus and {D}ehn surgery,''
  \href{http://dx.doi.org/10.1073/pnas.1018918108}{{\em Proc. Natl. Acad. Sci.
  U.S.A.} {\bfseries 108} (05, 2011) 8106--13}.

\bibitem{LekiliPolishchuk}
Y.~Lekili and A.~Polishchuk, ``Arithmetic mirror symmetry for genus 1 curves
  with $n$ marked points,''
  \href{http://dx.doi.org/10.1007/s00029-016-0286-2}{{\em Sel. Math. New Ser.}
  {\bfseries 23} (2016) 1851–--1907}.

\bibitem{Lekili2017AMC}
Y.~Lekili and A.~Polishchuk, ``{A modular compactification of
  $\mathcal{M}_{1,n}$ from $A_{\infty}$-structures},'' {\em J. Reine Angew.
  Math.} {\bfseries 2019} no.~755, (2017) 151 -- 189.

\bibitem{Auslanderorders}
Y.~Lekili and A.~Polishchuk, ``Auslander orders over nodal stacky curves and
  partially wrapped {F}ukaya categories,''
  \href{http://dx.doi.org/10.1112/topo.12064}{{\em J. Topol.} {\bfseries 11}
  (2017) 615--644}.

\bibitem{Lekili2019}
Y.~Lekili and A.~Polishchuk, ``{Derived equivalences of gentle algebras via
  {F}ukaya categories},''
  \href{http://dx.doi.org/10.1007/s00208-019-01894-5}{{\em Math. Ann.} (2020)
  187–--225}.

\bibitem{2018arXiv180604345L}
Y.~{Lekili} and K.~{Ueda}, ``{Homological mirror symmetry for {M}ilnor fibers
  via moduli of $A_\infty$-structures},'' {\em Preprint, arXiv: 1806.04345} .

\bibitem{2020arXiv200407374L}
Y.~{Lekili} and K.~{Ueda}, ``{Homological mirror symmetry for Milnor fibers of
  simple singularities},'' {\em Algebraic Geometry} (2021) .

\bibitem{Orlov2009}
D.~Orlov, {\em Derived categories of coherent sheaves and triangulated
  categories of singularities},
  \href{http://dx.doi.org/10.1007/978-0-8176-4747-6_16}{pp.~503--531}.
\newblock Birkh{\"a}user Boston, Boston, 2009.

\bibitem{pinkham1974deformations}
H.~Pinkham, {\em {Deformations of algebraic varieties with
  $\mathbb{G}_m$-action}}.
\newblock No.~v. 20-21. Soci{\'e}t{\'e} math{\'e}matique de France, 1974.

\bibitem{polishchuk2017}
A.~Polishchuk, ``Moduli of curves as moduli of ${A}_{\infty}$-structures,''
  \href{http://dx.doi.org/10.1215/00127094-2017-0019}{{\em Duke Math. J.}
  {\bfseries 166} no.~15, (10, 2017) 2871--2924}.

\bibitem{BSMF_2000__128_1_103_0}
P.~Seidel, ``Graded {L}agrangian submanifolds,''
  \href{http://dx.doi.org/10.24033/bsmf.2365}{{\em Bull. Soc. Math. France}
  {\bfseries 128} no.~1, (2000) 103--149}.

\bibitem{HMSQuartic}
P.~Seidel, ``Homological mirror symmetry for the quartic surface,''
  \href{http://dx.doi.org/10.1090/memo/1116}{{\em Mem. Amer. Math. Soc.}
  {\bfseries 236} (10, 2003) }.

\bibitem{seidel2008}
P.~Seidel, ``A biased view of symplectic cohomology,'' {\em Current
  Developments in Mathematics} {\bfseries 2006} (2008) 211--253.

\bibitem{SeidelBook}
P.~Seidel, \href{http://dx.doi.org/10.4171/063}{{\em Fukaya categories and
  {P}icard-{L}efschetz theory}}.
\newblock Zurich Lectures in Advanced Mathematics. European Mathematical
  Society (EMS), Z\"{u}rich, 2008.

\bibitem{Shipmandg}
I.~Shipman, ``{A geometric approach to {O}rlov's theorem},''
  \href{http://dx.doi.org/10.1112/S0010437X12000255}{{\em Compos. Math.}
  {\bfseries 148} (12, 2012) 1365--1389}.

\bibitem{Smyth}
D.~Smyth, ``Modular compactifications of the space of pointed elliptic curves
  {I},'' \href{http://dx.doi.org/10.1112/S0010437X10005014}{{\em Compos. Math.}
  {\bfseries 147} (2011) 877 -- 913}.

\bibitem{Weightedprojectivelines}
A.~Takahashi, ``Weighted projective lines associated to regular systems of
  weights of dual type,'' {\em Adv. Stud. Pure Math.} {\bfseries 59} (01, 2010)
  371--388.

\bibitem{2006math......4361U}
K.~{Ueda}, ``{Homological mirror symmetry and simple elliptic singularities},''
  {\em Preprint, arXiv: math/0604361} .

\bibitem{Hyperplanesections}
K.~Ueda, ``Hyperplane sections and stable derived categories,''
  \href{http://dx.doi.org/10.1090/S0002-9939-2014-12124-1}{{\em Proc. Am. Math.
  Soc} {\bfseries 142} no.~9, (07, 2012) 3019--3028}.

\end{thebibliography}\endgroup
\bibliographystyle{utcapsor2}

\end{document}